\documentclass[default,12pt]{sn-jnl}

\usepackage{amssymb,latexsym,amsmath}
\usepackage{epsfig}
\usepackage{caption}
 
 \usepackage{amsmath}
\usepackage{amsfonts}
\usepackage{array}
\usepackage{dsfont}
\usepackage{amssymb}
\usepackage{amsthm}

\usepackage{graphicx}%
\usepackage{multirow}%
\usepackage{amsmath,amssymb,amsfonts}%
\usepackage{amsthm}%
\usepackage{mathrsfs}%
\usepackage[title]{appendix}%
\usepackage{xcolor}%
\usepackage{textcomp}%
\usepackage{manyfoot}%
\usepackage{booktabs}%
\usepackage{algorithm}%
\usepackage{algorithmicx}%
\usepackage{algpseudocode}%
\usepackage{listings}%
\theoremstyle{thmstyleone}%
\newtheorem{theorem}{Theorem}%  meant for continuous numbers
\newtheorem{proposition}{Proposition}% to get separate numbers for theorem and proposition etc.

\theoremstyle{thmstyletwo}%

\newtheorem{example}{Example}%
\newtheorem{remark}{Remark}%

\theoremstyle{thmstylethree}%

\newtheorem{assumption}{Assumption}

\newtheorem{lemma}{Lemma}

\def\B{{\mathcal B}}

\def\P{{\mathcal P}}

\def\M{{\mathcal M}}

\newcommand{\adk}[1]{{\color{black} #1}}

\usepackage[utf8]{inputenc}

\begin{document}

\title[Finite Approximations for Mean-field Type Multi-Agent Control and Their Near Optimality]{Finite Approximations for Mean-field Type Multi-Agent Control and Their Near Optimality }


\author[1]{\fnm{Erhan} \sur{Bayraktar}}\email{erhan@umich.edu}

\equalcont{ E. Bayraktar is partially supported by the National Science Foundation under grant DMS-2106556 and by the Susan M. Smith chair.}

\author[2]{\fnm{Nicole} \sur{ B\"{a}uerle}}\email{nicole.baeuerle@kit.edu}

\author*[3]{\fnm{Ali Devran} \sur{Kara}}\email{alikara@umich.edu}

\affil[1]{\orgdiv{Department of Mathematics}, \orgname{University of Michigan}, \orgaddress{ \city{Ann Arbor},  \state{MI}, \country{USA}}}

%\affil[2]{\orgdiv{Department of Mathematics}, \orgname{University of Michigan}, \orgaddress{ \city{Ann Arbor},  \state{MI}, \country{USA}}}

\affil[2]{\orgdiv{Department of Mathematics}, \orgname{Karlsruhe Institute of Technology}, \orgaddress{ \city{Karlsruhe}, \country{Germany}}}

\affil[3]{\orgdiv{Department of Mathematics}, \orgname{Florida State University}, \orgaddress{ \city{Tallahassee},  \state{FL}, \country{USA}}}


\abstract{We study a multi-agent mean-field type control problem in discrete time where the agents aim to find a socially optimal strategy and where the state and action spaces for the agents are assumed to be continuous. The agents are only weakly coupled through the distribution of their state variables. The problem in its original form can be formulated as a classical Markov decision process (MDP), however, this formulation suffers from several practical difficulties. In this work, we attempt to overcome the curse of dimensionality, coordination complexity between the agents, and the necessity of perfect feedback collection from all the agents (which might be hard to do for large populations.)

 We provide several approximations: we establish the near optimality of the action and state space discretization of the agents under standard regularity assumptions for the considered formulation by constructing and studying the measure valued MDP counterpart for finite and infinite population settings. It is a well known approach to consider the infinite population problem for mean-field type models, since it provides symmetric policies for the agents which simplifies the coordination between the agents. However, the optimality analysis is harder as the state space of the measure valued infinite population MDP is continuous (even after space discretization of the agents). Therefore, as a final step, we provide two further approximations for the infinite population problem: the first one directly aggregates the probability measure space, and requires the distribution of the agents to be collected and mapped with a nearest neighbor map, and the second method approximates the measure valued MDP through the empirical distributions of a smaller sized sub-population, for which one only needs keep track of the mean-field term as an estimate by collecting the state information of a small sub-population. For each of the approximation methods, we provide provable regret bounds.

}

\maketitle

%\tableofcontents

\section{Introduction}\label{intro}
The goal of this paper is to develop finite state and action approximation methods for mean-field type multi-agent control problems in discrete time. We will focus on a formulation where the agents depend on each other only through the mean-field term. i.e. the distribution of the agents' states. Furthermore, we will assume that the dynamics are symmetric for different agents that they do not depend on the identity of the agent but only their state, action and the mean-field term. This formulation is motivated for systems where the effect of a single agent on the whole system is minimal.

The dynamics for the model are presented as follows: suppose $N$ agents (decision-makers or controllers) act in a cooperative way to minimize a common objective function. The agents share a common state and an action space denoted by $\mathds{X}\subset \mathds{R}^l$ and $\mathds{U}\subset{\mathds{R}}^m$ for some $l,m<\infty$. For any time step $t$, and for some agent $i\in\{1,\dots,N\}$ we have that
\begin{align*}
x^i_{t+1}=f(x_t^i,u_t^i,\mu_{\bf x_t},w_t^i,w_t^0)
\end{align*}
for a measurable function $f$, where $\{w_t^i\}$ denotes the i.i.d. idiosyncratic noise process, and $\{w_t^0\}$ denotes the  i.i.d. common  noise process. \adk{Furthermore, $\mu_{\bf x}\in\mathcal{P}_N(\mathds{X})$ denotes the empirical distribution of the agents on the state space $\mathds{X}$ such that for a given joint state of the team of agents ${\bf x}:=(x^1,\dots,x^N)\in \mathds{X}^N$
\begin{align*}
\mu_{\bf x}(\cdot):=\frac{1}{N}\sum_{i=1}^N\delta_{x^i}(\cdot)
\end{align*}
where $\delta_{x^i}$ represents the Dirac measure centered at $x^i$.
Throughout this paper, we use the notation
\begin{align*}
\mathds{X}^N:=\underbrace{\mathds{X}\times\dots\times\mathds{X}}_{\text{ $N$ times}}
\end{align*}
to denote the space of all joint state variables of the team
equipped with the product topology on $\mathds{X}^N$.
We further define $\mathcal{P}_N(\mathds{X})$, the set of all  empirical measures on $\mathds{X}$ constructed using sequences of $N$ states in $\mathds{X}$, such that
\begin{align*}
\mathcal{P}_N(\mathds{X}):=\{\mu_{\bf x}: {\bf x}=(x^1,\dots,x^N)\in\mathds{X}^N\}.
\end{align*}
Note that $\mathcal{P}_N(\mathds{X})\subset \P(\mathds{X})$ where $\P(\mathds{X})$ denotes the set of all probability measures on $\mathds{X}$ equipped with the weak convergence topology.}

Equivalently, for a given realization of the common noise, $w_t^0$, the next state of the agent $i$ is determined by some stochastic kernel:
\begin{align}\label{kernel_common_noise}
X^i_{t+1}\sim\mathcal{T}^{w^0_t}(\cdot|x_t^i,u_t^i,\mu_{\bf x_t})
\end{align}
where $\mathcal{T}^{w^0_t}$ specifies the conditional probability distribution of the next state given the current state, action, empirical distribution, and the realization of the common noise. The randomness in the transition is due to the idiosyncratic noise $w_t^i$.

At each time stage $t$, each agent receives a cost determined by a measurable stage-wise cost function $c:\mathds{X}\times\mathds{U}\times \mathcal{P}(\mathds{X})\to \mathds{R}$.  That is, if the state, action, and the empirical  distribution of the team are given by $x_t^i,u_t^i,\mu_{\bf x_t}$, then the agent receives the cost 
\begin{align*}
c(x_t^i,u_t^i,\mu_{\bf x_t}).
\end{align*}
For the remainder of the paper, by an abuse of notation, we will occasionally represent the dynamics in terms of the collective state and action variables of the team, ${\bf x}=(x^1,\dots,x^N)$, and ${\bf u}=(u^1,\dots,u^N)$, and the collective noise variables ${\bf w}=(w^0,w^1,\dots,w^N)$. Under this notation, the dynamics can be expressed compactly as:
\begin{align}\label{col_dyn}
{\bf x_{t+1}}=f({\bf x_t,u_t,w_t}).
\end{align}
In the initial formulation, every agent is assumed to have full knowledge of the state and action variables of every other agent. We define an admissible policy for an agent $i$, as a sequence of functions $\gamma^i:=\{\gamma^i_t\}_t$, where $\gamma^i_t$ is a $\mathds{U}$-valued (possibly randomized) function that is measurable with respect to the $\sigma$-algebra generated by the information
\begin{align}\label{info}
I_t=\{{\bf x}_0,\dots,{\bf x}_t,{\bf u}_0,\dots,{\bf u}_{t-1}\}.
\end{align}
Accordingly, an admissible {\it team} policy, is defined as $\gamma:=\{\gamma^1,\dots,\gamma^N\}$, where each $\gamma^i$ is an admissible policy for the agent $i$. In other words, agents share full information.

The objective of the agents is to minimize the following cost function for some initial team state ${\bf x}_0\in\mathds{X}^N$
\begin{align*}
J^N_\beta({\bf x}_0,\gamma)=\sum_{t=0}^{\infty}\beta^tE_\gamma\left[{\bf c}({\bf x}_t,{\bf u}_t)\right]
\end{align*}
where $0<\beta<1$ is some discount factor and 
\begin{align*}
{\bf c}({\bf x}_t,{\bf u}_t):=\frac{1}{N}\sum_{i=1}^Nc(x^i_t,u^i_t,\mu_{{\bf x}_t}).
\end{align*}
The optimal cost is defined by
\begin{align}\label{opt_cost}
J_\beta^{N,*}({\bf x}_0):=\inf_{\gamma\in\Gamma}J^N_\beta({\bf x}_0,\gamma)
\end{align}
where $\Gamma$ denotes the set of all admissible team policies.

We note that this information structure will be our benchmark to evaluate the performance of the approximate solutions that will be presented in the paper. In other words, the value function that is achieved when the agents share full information will be taken to be our reference point for simpler information structures. 
\subsection{Problem Motivation}
The team problem we have introduced can be modeled as a standard Markov decision process (MDP) with a central controller, where the state variable is the collective state of all agent, i.e. ${\bf{x}}\in\mathds{X}^N$. Therefore, in principle,  standard dynamic programming based approaches  such as value iteration, policy iteration can be used to solve the optimal control problem. However, there are some practical challenges for both deriving and implementing the optimal solution.

The first and the main challenge is due to the curse of dimensionality since the state space $\mathds{X}^N$, and the action space $\mathds{U}^N$ become hard to deal with especially when $\mathds{X}$, and $\mathds{U}$ are continuous. 

The second challenge is due to the coordination of the agents. Even if the optimal policy is known, implementation at the agent level is not easy. In general, the agents might have to follow asymmetric policies by coordinating with each other, which becomes increasingly difficult as the number of agents grows.

Lastly, the observation of the feedback variable is a difficult task as each agent needs to have access to the state variable of all other agents, to implement the optimal policy. This requirement is often impractical or infeasible in large systems due to communication and observation constraints.

Our goal in this paper is to propose approximate solution, policy implementation, and feedback collection methods. Furthermore, we will analyze the regret bounds and prove the near optimality  of the presented approximations.

%\subsection{A New Class of Policies}
%In this section, we will show that the policies can be restricted to a subset of admissible team policies without loss of optimality. We first note that a team policy at any given time $t$, induces a $\mathds{U}^N$ valued mapping which is measurable with respect to the $\sigma$-algebra generated by the information variables (\ref{info}).

%Let us denote the set of stationary and Markov team policies by $\Gamma_{M}$.  For any $\gamma\in\Gamma_M$, $\gamma=\{g,g,\dots\}$, that is the policies are stationary and $g$ is measurable with respect to the  $\sigma$-algebra generated by the latest state variable ${\bf x}_t$, at any given time $t$. By an abuse of notation, we will use $\gamma$ for the control functions induced the stationary and Markov team policies at any time $t$.

%We now focus on a further subset of stationary and Markov team policies $\Gamma_M$.  We denote the new policies by $\hat{\Gamma}$ such that
%\begin{align*}
%\hat{\Gamma}:=\{\gamma\in\Gamma_M|\mu_{\gamma({\bf x})}=\mu_{\gamma({\bf x'})} \text{ for all } {\bf{x,x'}}\in \mathds{X}^N \text{ s.t. } \mu_{\bf x}=\mu_{\bf x'} \}
%\end{align*}
%In other words, if ${\bf{x,x'}}\in \mathds{X}^N$ have the same empirical distribution, i.e. $\mu_{\bf x}=\mu_{\bf x'}$, we then require the control actions induced by the policy $\gamma\in\hat{\Gamma}$,  ${\bf u}=\gamma({\bf x})$ and ${\bf u'}=\gamma({\bf x'})$ to have the same empirical distribution, i.e.  $\mu_{\gamma({\bf x})}=\mu_{\gamma({\bf x'})}$.

\subsection{Literature Review and Contributions}
The mean-field framework  is used to study multi-agent problems where the population is homogeneous and weakly interacting.  In this setup, the dynamics of each agent depends on the other agents only through the empirical distribution of  the population. Depending on the cost structure of the problem, this framework is  referred to as mean-field game  theory when  agents are competitive and as mean-field control when the agents are cooperative.  The key observation in both settings is that for large populations (as the number of agents approaches to infinity), the decentralized problem can be reduced to a centralized decision problem, that is effectively analyzed through the perspective of a representative agent. % For the game problems, this reduction is extremely useful to find equilibrium solutions which is quite challenging with classical multi-agent techniques. For cooperative settings, although the problem can always be seen as centralized, if every agent knows every other agents information, with the limit reduction,  information structure of the problem is greatly simplified.

Mean-field game theory has been introduced independently by \cite{huang2006large} and \cite{lasry2007mean}. The area has gained a lot of attention after the first introduction. We refer the reader to \cite{gomes2014mean,carmona2013probabilistic,bensoussan2013mean,tembine2013risk,huang2007large,anahtarci2022q,elie2020convergence,fu2019actor,guo2019learning, perrin2020fictitious,subramanian2019reinforcement,saldi2018markov, saldi2019approximate} and references in them for mean-field game theory studies both in continuous and discrete time, studying different models and cost structures including computational and learning related methods.

For multi-agent control problems, the solution, in general, is intractable except under specific information structures among the agents. Mean-field type team problems, where the agents are only related through the mean-field term, is one of these special cases. For team control problems, where the agents are cooperative and work together to minimize (or maximize) a common cost (or reward) function, majority of  studies focus on the continuous time dynamics of the limit problem, see \cite{bayraktar2018randomized,djete2022mckean, lauriere2014dynamic,pham2017dynamic,carmona2021convergence,germain2022numerical,bayraktar2021mean,bayraktar2021solvability, sanjari2022optimality, sanjari2021optimal} and references therein for the study of dynamic programming principle, learning methods, and justification of the exchangeability of agents for large (possibly infinite) agent team settings.  The papers  \cite{lacker2017limit,fornasier2019mean,djete2022mckean1} provide the justification for studying the centralized limit problem by rigorously connecting the large population decentralized setting and the infinite population limit problem.

In this paper, we will study multi-agent mean-field control problems in discrete time. For this setting,  \cite{motte2022quantitative,motte2022mean} study mean-field control for infinite population problems under so called open loop controls where the controller can use the realization of the noise process to decide on actions. Furthermore the authors rigorously prove the connection between the limit problem and the finite population problem with a rate of convergence. \cite{gu2021mean,gu2019dynamic} study dynamic programming principle and solutions to the limit (infinite population) problem, and Q-learning methods. Similarly, \cite{carmona2019model} studies different classes of policies that achieve optimal performance for the limit problem and focuses on Q-learning of the problem after establishing the optimality of randomized feedback policies for the agents. \cite{pasztor2021efficient} studies mean-field control problem of infinite populations with additive noise models without common noise, and propose a model based learning algorithm using the class of Lipschitz continuous control policies. \cite{wang2020breaking} studies a model-free fitted Q-learning algorithm for the mean-field control problem with infinitely many agents. The authors focus on the case where the distribution of agents, i.e. the mean-field term is only observed through a sub-sample population, and thus they aim to tackle the curse of many agents for finite state and action spaces. \cite{angiuli2022unified} is another work focusing on model free reinforcement learning of mean-field control and game problems. The authors provide connections between learning for mean-field game and control problems under so called asymptotic and stationary formulations using a two-time scale approach.
Finally, \cite{bauerle2021mean} focuses on mean-field control setting for both finitely many and infinitely many agents problems by constructing a measure valued MDP for both cases. The paper rigorously establishes the connection of the finite population problem to the limit problem. Furthermore, the author studies the implication of these findings to the infinite horizon average cost problems.

In our paper, we aim to provide several approximations for the discrete time mean-field control problems. Our focus is on models where agents operate in continuous state and action spaces with general non-linear dynamics influenced by common noise. We do not put apriori assumptions on the set of optimal policies. We begin by working directly with finite population  setup, by constructing a measure valued MDP for the finite population as shown in \cite{bauerle2021mean}. This approach demonstrates that optimal performance can be achieved if agents have access to the (finite population) empirical distributions of the agents.  However, even after reducing the control problem to its  measure valued  counterpart, if the original state space is continuous, the state space of the measure valued problem will be a continuous state space. Hence, we then further approximate the problem by discretizing the state space and creating a measure valued MDP defined on  a finite set, and we provide provable error bounds resulting from the discretization. 

The solution of the finite population measure valued MDP provides a solution where the agents need to implement asymmetric policies, even though the agents are exchangeable. This asymmetry introduces coordination challenges, making the execution of optimal solutions complex. To break the correlation and the weak coupling between the agents, we also analyze the infinite population control problem. The limit problem can also be represented as a measure valued problem. This formulation can then be viewed as the control of the marginal distribution flow of a representative agent. This approach yields a symmetric policy for all agents which overcomes the coordination difficulty. We analyze the error bounds due to space discretization for the infinite population problem as well.

For the infinite population problem, however, the state space of the measure-valued MDP remains continuous-valued after discretizing the agent space. Hence, we  provide further approximations by $(i)$ further aggregating the state space of the measure-valued MDP into coarser representations and $(ii)$ through a sub-population sampling method. The second method uses the distribution of a smaller population to estimate the mean-field of the original population and thus it also addresses the feedback observation difficulty by reducing the dependence on full-population feedback in large-scale systems.

We now summarize our contributions in detail:

{\bf Contributions: } Our main goal in this paper is to develop approximate solution methods. We will develop the approximation techniques in several steps. As introduced in the problem formulation, the original team problem can be modeled as a centralized MDP with state space $\mathds{X}^N$, where $N$ is the number of agents. 
\begin{itemize}
\item In Section \ref{meas_valued_sec}, we show that the team problem can be modeled as a centralized measure valued MDP, and the state space can be reduced to $\P_N(\mathds{X})$, the set of all probability measures on $\mathds{X}$ that can be represented as the distributions of $N$ agents, without loss of optimality. We show that the optimal policies for agents can be realized as randomized policies for this formulation. This step can be thought of as removing the identities of the agents using the exchangeability of them and only focusing on which parts of the state space agents live, rather than keeping track of states of agents with agent identities. Although, this is a significant dimension reduction, the solution is still hard to compute with $\P_N(\mathds{X})$ being a continuous space. Furthermore, the optimal agent level policies are asymmetric, hence the coordination problem is still present.
\item Towards the dimension reduction purpose, in Section \ref{finite_model_sec}, we show that by discretizing the state space $\mathds{X}$, and separating $\mathds{X}$ into $M$ disjoint subsets, one can construct an MDP whose state space is finite with size $\frac{(M+N-1)!}{(N-1)!M!}$. This step can be thought as only counting the number of agents in the chosen subsets of the state space $\mathds{X}$, rather than looking at every point in the state space. We show that the policies constructed for this approximate model are nearly optimal for large $M$, if the dynamics and the cost functions  satisfy certain regularity properties.
\item  The approximation in Section \ref{finite_model_sec} results in a finite model, however, for large number of agents, the problem can be intractable, and the coordination of the agents remains unsolved. In Section \ref{inf_pop}, we present a control model by taking $N\to\infty$. The resulting models have the state space $\P(\mathds{\hat{X}})$, that is the set of all probability measures on the finite set $\hat{\mathds{X}}$. It has been shown in the literature that this model approximates the team problem when the number of agents tends to be very large, and the solution of the infinite limit can be used by every agent as a symmetric policy, hence, the coordination problem is solved.  However, the dimension of the problem increases from the set of all empirical measures to the set of all  probability measures.
\item Finally, in Section \ref{large_N}, we present approximate solutions for the infinite population problem, by discretizing the set of probability measures using the set of all probability measures on the discretized space that can be represented as the empirical distribution of $n$ agents, where $n$ is manageable. Hence, the approximate solution deals with the curse of dimensionality as well as the coordination complexity. Furthermore, we show that, for the application of the solution, one can keep track of a small portion of the agents as the feedback variable, instead of observing the perfect feedback variable. Thus, we deal with the feedback observation challenge. We also prove the near optimality of the approximations by proving regret upper bounds. 
\end{itemize}

\section{Measure Valued Centralized MDP Construction}\label{meas_valued_sec}

In this section, we will define a Markov decision process, for the distribution of the agents, where the control actions are the joint distribution of the state and action vectors of the agents. 

We let the state space to be  $\mathcal{Z}=\P_N(\mathds{X})$ which is the set of all distributions on $\mathds{X}$ that can be constructed using the joint states of the agents. The state in this form can be thought of as an empirical measure which is constructed by the joint state of the agents. We equip $\mathcal{Z}=\P_N(\mathds{X})\subset \P(\mathds{X})$ with the weak convergence topology.

The admissible set of actions for some state $\mu\in\mathcal{Z}$, is denoted by $U(\mu)$, where
\begin{align}\label{add_act}
U(\mu)=\{\Theta\in \P_N(\mathds{X}\times\mathds{U})|\Theta(\cdot,\mathds{U})=\mu(\cdot)\}.
\end{align}
\adk{We define $\P_N(\mathds{X}\times\mathds{U})$, the set of all joint distributions on $\mathds{X}\times\mathds{U}$ of $N$ agents, in a similar way to $\P(\mathds{X})$ such that
\begin{align*}
\P_N(\mathds{X}\times\mathds{U}):=\{\mu_{\bf(x,u)}(\cdot): {\bf (x,u)}\in\mathds{X}^N\times\mathds{U}^N\}.
\end{align*}
For any joint state-action pair $ {\bf (x,u)}\in\mathds{X}^N\times\mathds{U}^N$, we define the associated empirical measure $\mu_{\bf (x,u)}$ as
\begin{align}\label{sa_emp}
\mu_{\bf (x,u)} (\cdot):= \frac{1}{N}\sum_{i=1}^N\delta_{(x^i,u^i)}(\cdot).
\end{align}}
The above implies that the set of actions for a state $\mu$, is the set of all joint distributions on $\mathds{X}\times\mathds{U}$ of $N$ agents whose marginal on $\mathds{X}$ coincides with $\mu$.  

We equip the state space $\mathcal{Z}$, and the action sets $U(\mu)$, with the first order Wasserstein distance $W_1$.

In order to  define the transition model for this centralized MDP, we note that the distributions of the agents' states of the original team problem induces a controlled Markov chain. In particular, for some set $B\in \B(\mathcal{Z})$, we can write
\begin{align}\label{kern_def}
&Pr(\mu_{t+1}\in B|\mu_t,\dots\mu_0,\Theta_t,\dots,\Theta_0)\nonumber\\
&\qquad=\int_{{\bf x_t,u_t}\in \mathds{X}^N\times\mathds{U}^N}Pr(\mu_{t+1}\in B|{\bf x_t,u_t})Pr({\bf dx_t,du_t}|\mu_t,\dots\mu_0,\Theta_t,\dots,\Theta_0).
\end{align}
For any $\mu_{(\bf x_t,u_t)}=\Theta_t$, and $\mu_{\bf x_t}=\mu_t$, the inside term can be written as
\begin{align*}
Pr(\mu_{t+1}\in B|{\bf x_t,u_t})=\int\mathds{1}_{\{\mu_{f({\bf x_t,u_t,\mu_{x_t},w_t})}\in B\}}Pr(d{\bf w_t})
\end{align*}
that is once we are given the idiosyncratic and the common noise realizations, the states and thus the incoming distribution of the agents is  fully determined.

%where $P(\cdot)$ is the probability measure governing the idiosyncratic and the common noise processes, and ${\bf w_t}$ is the noise vector with length $N+1$. 

\adk{If two pairs $({\bf x_t,u_t}),({\bf x'_t,u'_t})$ have the same distribution $\Theta_t$, i.e. if $\mu_{\bf (x_t,u_t)}=\mu_{\bf (x'_t,u'_t)}=\Theta_t$ then $({\bf x_t,u_t})$ can be obtained by permuting the indices of agents in $({\bf x'_t,u'_t})$. This is due to the definition of the empirical measures, and the invariance of the summation in (\ref{sa_emp}) against the permutation of agent indices.  Indeed, the empirical measure $\Theta_t$ depends only on the distribution of the states and actions, not their ordering.} We note further that the dynamics are identical for every agent and that the agents are exchangeable. Consider some ${\bf w_t}$, and ${\bf x_{t+1}}=f({\bf x_t,u_t,\mu_{\bf x_t}w_t})$, where $\mu_{\bf x_{t+1}}=\mu_{t+1}$. By reordering ${\bf w_t}$, one can define a new noise vector  ${\bf w'_t}$ such that ${\bf x'_{t+1}}=f({\bf x'_t,u'_t,\mu_{\bf x_t},w'_t})$. Here, ${\bf x'_{t+1}}$ is just a reordered version of ${\bf x_{t+1}}$, and in particular $\mu_{\bf x_{t+1}}=\mu_{\bf x'_{t+1}}$.

Since the idiosyncratic noises are identically distributed for every agent, as a result of the above discussion, for any two pairs $({\bf x_t,u_t}),({\bf x'_t,u'_t})$ with the same distribution $\Theta_t$, 
\begin{align*}
Pr(\mu_{t+1}\in B|{\bf x_t,u_t})=Pr(\mu_{t+1}\in B|{\bf x'_t,u'_t}), \quad \forall B\in\B(\P_N(\mathds{X}))
\end{align*}
where $\B(\P_N(\mathds{X}))$ is the Borel $\sigma$-algebra on $\P_N(\mathds{X})$. As a result, the integrand in (\ref{kern_def}) remains constant for every ${\bf x}_t,{\bf u}_t$ with $\mu_{\bf x_t}=\mu_t$ and $\mu_{\bf (x_t,u_t)}=\Theta_t$.
Therefore, the distributions of the agent states $\mu_t$, and of the joint state and actions $\Theta_t$ define a controlled Markov chain such that for all $B\in\B(\P_N(\mathds{X}))$
\begin{align}\label{trans_measure}
&Pr(\mu_{t+1}\in B|\mu_t,\dots\mu_0,\Theta_t,\dots,\Theta_0)=Pr(\mu_{t+1}\in B|\mu_t,\Theta_t)\nonumber\\
&:=\eta(B|\mu_t,\Theta_t)\\
&=Pr(\mu_{t+1}\in B|{\bf x_t,u_t}), \text{ for any } ({\bf x_t,u_t}): \mu_{({\bf x_t,u_t})}=\Theta_t\nonumber
\end{align}
where $\eta(\cdot|\mu,\Theta)\in \P(\P_N(\mathds{X}))$ is the transition kernel of the centralized measure valued MDP, which is induced by the dynamics of the team problem.

We define the stage-wise cost function $k(\mu,\Theta)$ by
\begin{align}\label{cost_measure}
k(\mu,\Theta):=\int c(x,u,\mu)\Theta(dx,du)=\frac{1}{N}\sum_{i=1}^Nc(x^i,u^i,\mu).%={\bf c(x,u)}
\end{align}

%In this section, we will define a Markov decision process, with a central controller, where the state space is $\P_N(\mathds{X})$, and the action space is $%\P_N(\mathds{X}\times\mathds{U})$. In other words, the central controller observes the empirical distribution of the agents $\mathds{X}$, say $\mu_{\bf x}$, and selects an empirical distribution on $\mathds{X}\times\mathds{U}$, say $Q(dx,du)$, as an action with the restriction that the marginal on $\mathds{X}$ is $\mu_{\bf x}$.

%One can define a stage-wise cost function $k:\P_N(\mathds{X})\times\P_N(\mathds{X}\times\mathds{U})\to \mathds{R}$ and a transition kernel $\eta(\cdot|\mu,Q)\in \P(\P_N(\mathds{X}))$ (\adk{need to define these properly}). Thus, we can have a centralized  MDP. 

Thus, we have an MDP with state space $\mathcal{Z}$, action space $\cup_{\mu\in\mathcal{Z}}U(\mu)$, transition kernel $\eta$ and the stage-wise cost function $k$.

We define the set of admissible policies for this measure valued MDP as a sequence of functions $g=\{g_0,g_1,g_2,\dots\}$ such that at every time $t$, $g_t$ is measurable with respect to the $\sigma$-algebra generated by the information variables
\begin{align*}
I_t=\{\mu_0,\dots,\mu_t,\Theta_0,\dots,\Theta_{t-1}\}.
\end{align*}
%Furthermore, for every $f_t(\mu_t)=Q_t$, the first marginal of $Q_t$ has to be $\mu_t$. 
We denote the set of all admissible control policies by $G$ for the measure valued MDP.

In particular, we define the infinite horizon discounted expected cost function under a policy $g$ by
\begin{align*}
K^N_\beta(\mu_0,g)=E_{\mu_0}^{\eta,g}\left[\sum_{t=0}^\infty \beta^t k(\mu_t,\Theta_t)\right].
\end{align*}

We also define the optimal cost by
\begin{align}\label{opt_cost2}
K_\beta^{N,*}(\mu_0)=\inf_{g\in G}K^N_\beta(\mu_0,g).
\end{align}

\subsection{Equivalence between the Team Problem and the Measure Valued MDP}
In this section, we show that the optimal value function of the original team problem, $J_\beta^{N,*}$ (see (\ref{opt_cost})), and the optimal value function of the measure valued MDP, $K_\beta^{N,*}$ (see (\ref{opt_cost2})) are equal. Furthermore, an optimal policy designed for the measure valued MDP can be realized as a randomized policy for the team problem, which achieves the optimal performance.

Before the main result of this section, we present the set of assumptions that will be used frequently in the paper and  we present a key lemma:

\begin{assumption}\label{main_assmp}
\begin{itemize}
\item[i.)] $\mathds{X}$ and $\mathds{U}$ are compact.
\item[ii.)] $f$ is Lipschitz in $x,u,\mu_{\bf x}$ such that
\begin{align*}
|f(x,u,\mu_{\bf x},w^i,w^0)-f(x',u',\mu_{\bf x'},w^i,w^0)|\leq K_f \left(\|x-x'\|+\|u-u'\|+W_1(\mu_{\bf x},\mu_{\bf x'})\right)
\end{align*} 
for some $K_f<\infty$, uniformly in $w^i, w^0$ where $W_1$ is the first order Wasserstein distance.
\item[iii.)] $c$ is  Lipschitz in $x,u,\mu_{\bf x}$ such that
\begin{align*}
|c(x,u,\mu_{{\bf x}})-c(x',u',\mu_{{\bf x'}})|\leq K_c \left(\|x-x'\|+\|u-u'\|+W_1(\mu_{\bf x},\mu_{\bf x'})\right)
\end{align*}
for some $K_c<\infty$.
\end{itemize}
\end{assumption}

\begin{lemma}\label{val_const}
Under Assumption \ref{main_assmp}, we have that the optimal value function is constant over the initial states with the same distribution, i.e. if ${\bf{x_0,x_0'}}\in \mathds{X}^N$ have the same distribution, $\mu_{\bf x_0}=\mu_{\bf x_0'}$, then
\begin{align*}
J_\beta^{N,*}({\bf x_0})=J_\beta^{N,*}({\bf x_0'}).
\end{align*}
%\item[ii.] The optimal team policies can be chosen from $\hat{\Gamma}$ without loss of optimality, i.e. for any ${\bf x}\in\mathds{X}^N$
%\begin{align*}
%\inf_{\gamma\in\Gamma}J_\beta({\bf x},\gamma)=\inf_{\hat{\gamma}\in\hat{\Gamma}}J_\beta({\bf x},\hat{\gamma})
%\end{align*}
%\end{itemize}
\end{lemma}

\begin{proof}
The proof can be found in Apendix \ref{val_const_sec}.
\end{proof}

\begin{theorem}\label{thm_meas}
Under Assumption \ref{main_assmp}, for any ${\bf x}_0$ that satisfies $\mu_{{\bf x}_0}=\mu_0$, that is for any ${\bf x}_0$ with distribution $\mu_0$, we have that
\begin{itemize}
\item[i.)]
\begin{align*}
K_\beta^{N,*}(\mu_{0})=J_\beta^{N,*}({\bf x}_0).
\end{align*}
%\item[ii.] For any stationary and Markov policy $g$ for the measure valued MDP, every agent can construct a policy $\gamma^i:\mathds{X}\times \P_N(\mathds{X})\to\mathds{U}$, that is each agent makes a decision by looking at their own states and the empirical distribution of the other agents. If we denote the resulting team policy  also by $\gamma$ with an abuse of notation, we have that
%\begin{align*}
%J^N_\beta({\bf x_0},\gamma)=K^N_\beta({\mu_0},g).
%\end{align*}

\item[ii.)] There exists a stationary and Markov optimal policy $g^*$ for the measure valued MDP, and  using $g^*$, every agent can construct a policy $\gamma^i:\mathds{X}\times \P_N(\mathds{X})\to\mathds{U}$ such that for $\gamma:=\{\gamma^1,\gamma^2,\dots,\gamma^N\}$
\begin{align*}
J^N_\beta({\bf x_0},\gamma)=J_\beta^{N,*}({\bf x_0}).
\end{align*}
\end{itemize}
\end{theorem}

\begin{proof}
The proof can be found in Appendix \ref{thm_meas_sec}. 
\end{proof}

\begin{remark}
The second part of the previous result is significant for reducing the complexity of the problem. An implication of the result is that, our search space for optimality reduces from $\mathds{X}^N$ to $\P_N(\mathds{X})$ (set of all distributions on $\mathds{X}$ with length $N$).
\end{remark}

\section{Finite Model Construction}\label{finite_model_sec}

\subsection{Action Space Discretization} We start by reducing the action space $\mathds{U}$ of agents to a finite set. 
We let  $\hat{\mathds{U}} = \{u_{1},\ldots,u_{k}\} \subset \mathds{U}$ be a subset of the original action space and we define:

\begin{align}\label{quant_error_action}
L_{\mathds{U}}:=&\sup_{u\in\mathds{U}}\min_{\hat{u}\in\hat{\mathds{U}}}\left\|u-\hat{u}\right\|
\end{align}
where we use the Euclidean norm for $\mathds{U}\subset \mathds{R}^m$.
This represents the worst error bound that results from the discretization of the action space. 

One can define the team problem introduced in Section \ref{intro} in an identical way where the action space of the agents is replaced with $\hat{\mathds{U}}$. We denote the optimal value function of the team problem with finite action spaces by $\tilde{J}_\beta^{N,*}({\bf x}_0)$.

 Recall that the team problem in its original formulation can be seen as a centralized MDP when every agent has access to the state and action information of every other agent. Hence the following can be seen as an implication of \cite[Theorem 3.16]{SaLiYuSpringer}. However, we present a proof Appendix \ref{act_fin_proof} for completeness.
\begin{theorem}\label{act_fin}
Under Assumption \ref{main_assmp}, if $2K_f\beta<1$, we have
\begin{align*}
|\tilde{J}_\beta^{N,*}({\bf x}_0)-{J}_\beta^{N,*}({\bf x}_0)|\leq  \frac{K_c}{(1-2K_f\beta)(1-\beta)} L_{\mathds{U}}
\end{align*}
where $K_f$ and $K_c$ represent the Lipschitz coefficient of the dynamics function $f$ and the cost function $c$ respectively (see Assumption \ref{main_assmp}).
\end{theorem}
\adk{
\begin{assumption}
Since we can control the upper-bound in the last result under the assumption that $\mathds{U}$ is compact,  in the sequel, we assume that $\mathds{U}$ is finite. 
\end{assumption}}
\subsection{State Space Discretization}\label{finite_sec}
For the finite state space approximation, we start by choosing a collection of disjoint sets $\{B_i\}_{i=1}^M$ such that $\cup_i B_i=\mathds{X}$, and $B_i\cap B_j =\emptyset$ for any $i\neq j$. Furthermore, we choose a representative state, $\hat{x}_i\in B_i$, for each disjoint set. For this setting, we denote the new finite state space by
$\hat{\mathds{X}}:=\{\hat{x}_1,\dots,\hat{x}_M\}$. The mapping from the original state space $\mathds{X}$ to the finite set $\hat{\mathds{X}}$ is done via
\begin{align}\label{quant_map}
\phi(x)=\hat{x}_i \quad \text{ if } x\in B_i.
\end{align}
With an abuse of notation, we extend this definition to the state vector of the agents, such that $\phi({\bf x})$ denotes the mapping when $\phi$ is applied to every element of ${\bf x}=(x^1,\dots,x^N)$ separately.
Furthermore, we choose a weight measure 
\begin{align}\label{weight_meas}
\pi(\cdot)\in\P(\mathds{X}^N).
\end{align}
such that $\pi(B_{i_1}\times B_{i_2}\times\dots\times B_{i_N})>0$ for any $B_{i_1},B_{i_2},\dots,B_{i_N}$.
% such that $\pi(B_i)>0$ for all $i\in\{1,\dots,M\}$. We now define normalized measures using the weight measure on each separate quantization bin $B_i$ such that  
%\begin{align}\label{norm_inv}
%\hat{\pi}_{\hat{x}_i}(A):=\frac{\pi(A)}{\pi(B_i)}, \quad \forall A\subset B_i, \quad \forall i\in \{1,\dots,M\}
%\end{align}
 %that is $\hat{\pi}_{\hat{x}}$ is the normalized weight measure on the set $B_i$, $\hat{x}$ belongs to. In other words,
%\begin{align*}
%Pr(x\in A|\phi(x)=x_i)=\hat{\pi}_{\hat{x}_i}(A)  \quad \forall A\subset B_i, \quad \forall i\in \{1,\dots,M\}.
%\end{align*}

We also define the uniform quantization error that results from the discretization of the state space
\begin{align}\label{quant_error}
L_{\mathds{X}}:=&\max_i\sup_{x,x'\in B_i}\left\|x-x'\right\|
\end{align}
where we use the Euclidean norm for $\mathds{X}\subset\mathds{R}^l$.
\subsection{Measure Valued Finite MDP Construction}\label{measure_valued_MDP}
In this section, building on the discretization in Section \ref{finite_sec}, we will construct a measure valued finite MDP where the state space is $\P_N(\hat{\mathds{X}})$, and the action space is $\P_N(\hat{\mathds{X}}\times\mathds{U})$. In other words, the central controller observes the distribution of the agents on the finite space $\hat{\mathds{X}}$, say $\mu_{\hat{\bf x}}$, and selects a distribution on $\hat{\mathds{X}}\times\mathds{U}$ as an action with the restriction that the marginal on $\hat{\mathds{X}}$ is $\mu_{\hat{\bf x}}$.

Note that with a simple combinatorial analysis, it can be shown that the size of the state space $\P_N(\hat{\mathds{X}})$ is $\frac{(M+N-1)!}{(N-1)!M!}$ where $N$ is the number of agents, and $M$ is the number of quantization bins. Similarly, for a given state $\mu_{\hat{\bf x}}$, the set of all admissible actions has the size $M\frac{(|\mathds{U}|+N-1)!}{(N-1)!|\mathds{U}|!}$. Hence, the new state and action spaces are finite.

We enumerate the states of the finite state space $\P_N(\hat{\mathds{X}})$, and denote by $\hat{\mu}_j$ the jth state. We now show that the discretization in the space $\mathds{X}$ induces a discretization in $\P_N(\mathds{X})$, which creates $\P_N(\hat{\mathds{X}})$. We define the sets
\begin{align*}
A_j=\{\mu_{\bf x}: \mu_{\phi({\bf x})}=\hat{\mu}_j\},
\end{align*} 
that is, $A_j$ is the set of measures on $\mathds{X}$ for the vector states ${\bf x}$, which give the finite measure $\mu_j$ when they are discretized. 

Note that, these sets define a map from $\P_N(\mathds{X})$ to $\P_N(\hat{\mathds{X}})$, such that each $\mu\in\P_N(\mathds{X})$ is mapped to some $\hat{\mu}_j\in\P_N(\mathds{\hat{X}})$. We will denote this map also by  $\phi$ with an abuse of notation to avoid notational clutter. 

The next supporting result shows that the quantization error, $L_{\mathds{X}}$, for the space $\mathds{X}$, projects to the quantization error for the space $\P_N(\mathds{X})$, when it is metrized by the first order Wasserstein metric.
\begin{lemma}\label{quant_lem}
\begin{align*}
\sup_j\sup_{\mu,\nu\in A_j}W_1(\mu,\nu)\leq L_{\mathds{X}}.
\end{align*}
where $L_\mathds{X}$ is as defined in (\ref{quant_error}).
\end{lemma}
\begin{proof}
Let ${\bf x_1,x_2}$ be such that $\mu_{\phi({\bf x_1})}=\mu_{\phi({\bf x_2})}$. Note that since the discretized versions of ${\bf x_1}$ and ${\bf x_2}$ have the same distributions, then ${\bf x_1}$ and ${\bf x_2}$ have the exact same number of elements in every quantization bin of $\mathds{X}$. By reordering, and collecting the elements from the same quantization bin together:
\begin{align*}
W_1(\mu_{\bf x_1},\mu_{\bf x_2})=\sup_{Lip(h)\leq 1}\left|\int h d\mu_{\bf x_1}-\int h d\mu_{\bf x_2}\right|&\leq\sup_{Lip(h)\leq 1}\frac{1}{N}\sum_{i=1}^N|h(x^i_1)-h(x^i_2)|\\
&\leq L_\mathds{X}
\end{align*}
where the last step follows from the fact that $h$'s Lipschitz coefficient is smaller than 1, and $x^i_1, x^i_2$ belong to the same quantization bin of $\mathds{X}$ due to reordering. Furthermore, the quantization error on $\mathds{X}$ is upper bounded by $L_{\mathds{X}}$ by definition (see (\ref{quant_error})).
\end{proof}

We now construct the normalized weight measures for the subsets $A_j\subset \P_N({\mathds{X}})$, using the previously chosen measures $\pi$ (see (\ref{weight_meas})), such that given the finite measure is $\mu_j$, for $A \subset A_j$
\begin{align}\label{norm_meas}
Pr(\mu \in A | \mu\in A_j)=\frac{Pr(\{{\bf x}:\mu_{\bf x}\in A\})}{Pr(\{{\bf x}:\mu_{\phi(\bf x)}\})}=\frac{\pi(\{{\bf x}:\mu_{\bf x}\in A\})}{\pi(\{{\bf x}\mu_{\phi(\bf x)}=\mu_j\})}=:\hat{\pi}_j(A).
\end{align}
Using these normalized weight measures, we define the cost function and the transition model for the finite measure valued MDP.

We denote the cost function by $\hat{k}:\P_N(\hat{\mathds{X}})\times \P_N(\hat{\mathds{X}}\times\mathds{U})\to \mathds{R}$ such that for some $\hat{\mu}_j\in\P(\hat{\mathds{X}})$ and $\hat{\Theta}\in  \P_N(\hat{\mathds{X}}\times\mathds{U})$
\begin{align}\label{finite_cost}
\hat{k}(\hat{\mu}_j,\hat{\Theta}):=\int_{A_j} k(\mu,\Theta_\mu)\hat{\pi}_j(d\mu)
\end{align}
where $\hat{\pi}_j$ is the normalized weight measure for the set $A_j\subset \P_N({\mathds{X}})$ of $\hat{\mu}_j$, and $k$ is the cost function of the original measure valued MDP as defined in (\ref{cost_measure}). Furthermore, for $\hat{\Theta}(dx,du)=\gamma(du|x)\hat{\mu}_j(dx)$, we have $\Theta_\mu(dx,du)=\gamma(du|\phi(x))\mu(dx)$. %In other words, $\hat{\Theta}(dx,du)$, and $\Theta_\mu(dx,du)$ agree on their second marginals whereas the first marginal of $\hat{\Theta}(dx,du)$ is give by $\mu_j$ and the one of $\Theta_\mu(dx,du)$ is given by $\mu$

We denote the transition law of the finite model by $\hat{\eta}$ such that 
\begin{align}\label{finite_kernel}
\hat{\eta}(\hat{\mu}_i|\hat{\mu}_j,\hat{\Theta})=\int_{A_j}\eta(A_i|\mu,\Theta_\mu)\hat{\pi}_j(d\mu)
\end{align}
where $\eta$ is the transition kernel of the original measure valued MDP as defined in (\ref{trans_measure}).

In other words, to define the cost and the transitions of the finite model, we average over the cost and the transitions of the original model using the normalized weight measures.

%\begin{align*}
%Pr({\bf x}\in A|\phi({\bf x})=\hat{\bf x})=\hat{\pi}_{\bf \hat{x}}(A):=\hat{\pi}_{\hat{x}^1}(A_1)\times\dots \times\hat{\pi}_{\hat{x}^N}(A_N).
%\end{align*}
%If we denote the optimal policy, for this finite MDP, by $\psi:\P_N(\hat{\mathds{X}})\to \P_N(\hat{\mathds{X}}\times\mathds{U})$, then at every time step $t$, the central controller chooses a measure valued action, say $Q(dx,du)$, at the agent level this action is realized as a randomized policy by disintegrating $Q(dx,du)$ such that 
%\begin{align*}
%Q(dx,du)=\gamma(du|x)\mu_{\hat{\bf x}}(dx).
%\end{align*}
%Hence, an agent with a state $\hat{x}$ chooses its action from $\mathds{U}$, in a randomized way according to $\gamma(\cdot|\hat{x})$ which is constructed by disintegrating the action, $Q$, of the central controller.

For some $\hat{\mu}_j\in \P_N(\hat{\mathds{X}})$ and an admissible policy $\hat{g}$, the infinite horizon discounted cost function is defined as
\begin{align*}
\hat{K}_\beta^N(\hat{\mu}_j,\hat{g})=\sum_{t=0}^\infty \beta^t E[\hat{k}(\hat{\mu}_t,\hat{g}(\hat{\mu}_t))]
\end{align*}
where the expectation is defined with respect to the kernel $\hat{\eta}$ and the initial point $\hat{\mu}_j$.
We will denote the optimal value function of this finite measure valued MDP by
\begin{align*}
\hat{K}_\beta^{N,*}(\hat{\mu}_j):=\inf_{\hat{g}\in \hat{G}}\hat{K}_\beta^{N}(\hat{\mu}_j,\hat{g})
\end{align*}
for every $j$, where $\hat{G}$ is the set of all admissible policies for the finite measure valued MDP.

Note that we can extend the optimal value function $\hat{K}_\beta^{N,*}(\hat{\mu}_j)$ which is defined on $\P_N(\hat{\mathds{X}})$, over the set $\P_N(\mathds{X})$, by making it constant over the subsets $A_j$. Hence, we can compare the value functions of the original measure valued MDP problem, and the finite measure valued MDP problem. For the main result of this section, we need the following lemmas:

%\begin{lemma}\label{cost_kernel_bound}
%For any $\mu,\mu'\in\P_N(\mathds{X})$, there exists ${\bf x,x'}\in %%\mathds{X}^N$ such that $\mu_{\bf x}=\mu$, $\mu_{\bf x'}=\mu'$ and 
%\begin{align*}
%W_1(\mu,\mu')=\frac{1}{N}\sum_{i=1}^N|x_i-x_i'|.
%\end{align*}
%Furthermore, under Assumption \ref{main_assmp}, for the state vectors ${\bf %x,x'}$ and for some action vector ${\bf u}\in\hat{\mathds{U}}^N$ and for the empirical measures defined using these vectors such  that $\Theta=\mu_{\bf (x,u)},\Theta'=\mu_{\bf (x',u)}\in\P_N(\mathds{X}\times\hat{\mathds{U}})$ we have
%\begin{align*}
%|k(\mu,\Theta)-k(\mu',\Theta')|\leq 2K_c W_1(\mu,\mu')\\
%W_1(\eta(\cdot|\mu,\Theta),\eta(\cdot|\mu',\Theta'))\leq  2K_f W_1(\mu,\mu').
%\end{align*}
%\end{lemma}

\begin{lemma}\label{cost_kernel_bound}
Let $\mu,\mu'\in\P_N(\mathds{X})$ and  ${\bf u}\in\mathds{U}^N$. 
Under Assumption \ref{main_assmp}, there exist state vectors ${\bf x,x'}$  such  that $\mu_{\bf x}=\mu$, $\mu_{\bf x'}=\mu'$ and $\Theta=\mu_{\bf (x,u)},\Theta'=\mu_{\bf (x',u)}\in\P_N(\mathds{X}\times\mathds{U})$ and we have
\begin{align*}
|k(\mu,\Theta)-k(\mu',\Theta')|\leq 2K_c W_1(\mu,\mu')\\
W_1(\eta(\cdot|\mu,\Theta),\eta(\cdot|\mu',\Theta'))\leq  2K_f W_1(\mu,\mu').
\end{align*}
\end{lemma}
\begin{proof}
 The proof can be found in  Appendix \ref{cost_kernel_bound_proof}.
\end{proof}

\begin{lemma}\label{val_lip}
Under Assumption \ref{main_assmp}, if $2K_f\beta<1$ for any $\mu,\mu'\in\P_N({\mathds{X}})$
\begin{align*}
|K_\beta^{N,*}(\mu)-K_\beta^{N,*}(\mu')|\leq \frac{2K_c}{1-2K_f\beta}W_1(\mu,\mu').
\end{align*}
\end{lemma}

\begin{proof}
The proof can be found in Appendix \ref{val_lip_proof}.  Note that the result also follows from Theorem 4.1. (c) in \cite{hinderer2005lipschitz}.
\end{proof}

The following is a simple example which shows that the value function difference can be unbounded if $2K_f\beta\geq 1$:
\begin{example}
We consider a control-free (without loss of generality) team problem where $\mathds{X}=\mathds{R}$, and where the dynamics are deterministic such that for every agent $i$
\begin{align*}
x^i_{t+1}=f(x_t^i,u_t^i,\mu_t)=x^i_t+E_{\mu_t}\left[X\right],
\end{align*}
and the stage-wise cost function is given by
\begin{align*}
c(x_t^i,u_t^i,\mu_t)=x_t^i.
\end{align*}
We can check that
\begin{align*}
|f(x',u',\mu')-f(x,u,\mu)|&\leq |x-x'|+\left|\int x\mu(dx)-\int x\mu'(dx)\right|\\
&\leq |x-x'|+W_1(\mu,\mu')
\end{align*}
which implies that $K_f=1$.

We now consider two initial conditions
\begin{align*}
x_0^i=1,\qquad \hat{x}_0^i=0
\end{align*}
for all $i$. With this initial conditions, we get that $x_t^i=2^t$ and $\hat{x}_t^i=0$ for all $t$ and $i$. Furthermore, we can compute the value functions:
\begin{align*}
&K_\beta^{N,*}(\mu_0)=\sum_{t=0}^\infty (2\beta)^t\\
&K_\beta^{N,*}(\hat{\mu}_0)=\sum_{t=0}^\infty 0=0\\
\end{align*}
Hence, if $\beta\geq 1/2$, i.e if $2\beta K_f\geq 1$, the difference between the value functions stays unbounded.

\end{example}

The following result shows that the difference between the value functions of the finite measure valued model and the original measure valued model can be bounded:
\begin{proposition}\label{key_lem}
Under Assumption \ref{main_assmp}, if $2K_f\beta<1$, for any $\mu\in\P_N({\mathds{X}})$
\begin{align*}
\left|\hat{K}_\beta^{N,*}(\mu)-K_\beta^{N,*}(\mu)\right|\leq \frac{2K_c}{(1-\beta)(1-2\beta K_f)}L_{\mathds{X}}.
\end{align*}
\end{proposition}
\begin{proof}
The proof can be found in Appendix \ref{key_lem_proof}.
\end{proof}

%\begin{theorem}\label{main_thm}
%Under Assumption \ref{main_assmp}, if we denote the policy designed for the finite model by $\psi_{M}$ (note that this is the agent level randomized policy which is a function of $\mu_{\hat{\bf x}}$ and thus ${\bf x}$ and the agent's state $x$, hence it can be viewed as an admissible policy for the agent) and apply it in the original team problem, we will have
%\begin{align*}
%J_\beta({\bf x}_0,\psi_M)-J_\beta^*({\bf x}_0)\leq K(L_{\mathds{X}},K_f,K_c,\beta)
%\end{align*}
 %where $L_\mathds{X}$ is the uniform quantization error, $K_f$ is the Lipschitz constant of the function $f$, $K_c$ is the Lipschitz constant of the function $c$, and $\beta$ is the discount factor.
%\end{theorem}

%\begin{proof}[Proof Sketch]
%The proof will be a direct implication of Proposition \ref{key_lem}.
%\end{proof}

\subsection{Near Optimality of the Policy Constructed From the Finite Model}\label{planning_sec}
We first summarize the steps to construct the approximate policy. 
\begin{itemize}
\item Construct the transition kernel and the cost function for the measure valued centralized MDP using (\ref{trans_measure}) and (\ref{cost_measure}).
\item Choose a normalizing measure $\pi \in \P(\mathds{X})$ and quantization bins $\{B_i\}_{i=1}^M\subset \mathds{X}$, such that $\cup_{i=1}^M B_i=\mathds{X}$. Choose representative states for the quantization bins, say $\{\hat{x}_i\}_{i=1}^M$ and denote the finite space by $\hat{\mathds{X}}=\{\hat{x}_1,\dots,\hat{x}_M\}$. 
\item Construct a finite MDP based on the measure valued MDP using (\ref{finite_kernel}) and (\ref{finite_cost}).
\item Compute the optimal policy for the finite model, say $\hat{g}^N:\P_N(\hat{\mathds{X}})\to\P_N(\hat{\mathds{X}}\times \mathds{U})$.
\item Compute the policies for the agents using $\hat{g}^N$, say $\hat{\gamma}=\{\hat{\gamma}^1,\dots,\hat{\gamma}^N\}$
%\item Randomized policies, $\hat{\gamma}_{\hat{\mu}}(du|\hat{x})$, are functions defined on the finite space $\hat{\mathds{X}}$. To use them in the original problem, simply extend them over the original space $\mathds{X}$, by making them constant over the quantization bins.
\end{itemize}

\begin{theorem}\label{main_thm2}
Under Assumption \ref{main_assmp} if $2\beta K_f<1$, and if we denote the policies induced by the measure valued finite model by $\hat{\gamma}$ and apply it in the original team problem, we will have
\begin{align*}
J_\beta^N({\bf x}_0,\hat{\gamma})-J_\beta^{N,*}({\bf x}_0)\leq \frac{4K_c}{(1-\beta)^2(1-2\beta K_f)}L_{\mathds{X}}
\end{align*}
 where $L_\mathds{X}$ is the uniform quantization error.
\end{theorem}
\begin{proof}
We denote by $\mu_0=\mu_{\bf x_0}$, and $\hat{\mu}_0=\mu_{\hat{\bf x}_0}$. The agent level policies $\hat{\gamma}$ are obtained from the measure valued policy $\hat{g}$ (we drop the $N$ dependence for notation simplicity). Furthermore, we can disintegrate $\hat{g}(\hat{\mu}_0)$ -the optimal action for $\hat{\mu}_0$ for the discretized model- to write
\begin{align*}
\hat{g}(\hat{\mu}_0)=\hat{\gamma}_{\mu_0}(du|x)\hat{\mu}_0(dx).
\end{align*}
note that $\hat{\gamma}_{\mu_0}(du|x)$ is not the policy used by the agents, but just a conditional probability for the distribution of the agents on $\mathds{U}$.

Application of $\hat{\gamma}$ on the original model for the measure valued MDP induces the following distribution on $\mathds{X}\times\mathds{U}$
\begin{align*}
\hat{g}(\mu_0)=\hat{\gamma}_{\mu_0}(du|\phi(x))\mu_0(dx).
\end{align*}
Using Theorem \ref{thm_meas} (ii), we have that
\begin{align*}
&J_\beta^N({\bf x}_0,\hat{\gamma})=K^N_\beta(\mu_0,\hat{g})\\
&J_\beta^{N,*}({\bf x}_0)=K_\beta^{N,*}(\mu_0)
\end{align*}
%where $\mu_0=\mu_{\bf x_0}$.

We start with the following bound
\begin{align}\label{first_bound}
K_\beta^N(\mu_0,\hat{g})-K_\beta^{N,*}(\mu_0)\leq |K^N_\beta(\mu_0,\hat{g})-\hat{K}_\beta^{N,*}(\mu_0)|-|\hat{K}_\beta^{N,*}(\mu_0)-K_\beta^{N,*}(\mu_0)|.
\end{align}
The second term is bounded by Proposition \ref{key_lem}, so we focus on the first term. Recall that, for $\hat{K}_\beta^{N,*}(\mu_0)$, given that $\mu_{\phi({\bf x_0})}=\mu_j$, we have 
\begin{align*}
\hat{K}_\beta^{N,*}(\mu_0)=\int_{A_j}k(\mu',\Theta_{\mu'})\hat{\pi}_j(d\mu')+\beta\int_{A_j}\int_{\mu_1}\hat{K}_\beta^{N,*}(\mu_1)\eta(d\mu_1|\mu',\Theta_{\mu'})\hat{\pi}_j(d\mu'),
\end{align*}
where for $\Theta_{\mu'}=g^N(\mu')$, and thus, $\Theta_{\mu'}(dx,du)=\hat{\gamma}(du|\phi(x))\mu'(dx)$.

The Bellman equation for $K^N_\beta(\mu_0,\hat{g})$ can be written as
\begin{align*}
K^N_\beta(\mu_0,\hat{g})&=k(\mu_0,\hat{g}(\mu_0))+\beta\int K^N_\beta(\mu_1,\hat{g})\eta(d\mu_1|\mu_0,\hat{g}(\mu_0)).
\end{align*}
For $\mu_0\in A_j$ and for any $\mu'\in A_j$, since $\hat{g}$ uses the discretized versions of the state variables, and since $\mu$ and $\mu_j$ belong to the same quantization bin, we can find state and action vectors  ${\bf x,x',u}$ in accordance with Lemma \ref{cost_kernel_bound} such that $\mu_{\bf (x,u)}=\hat{g}(\mu_0)$ and  $\mu_{\bf (x',u)}=\hat{g}(\mu')$. Hence, by Lemma \ref{cost_kernel_bound}
\begin{align*}
|k(\mu_0,\hat{g}(\mu_0))-k(\mu',\hat{g}(\mu'))|\leq 2K_c W_1(\mu_0,\mu')\\
W_1(\eta(\cdot|\mu_0,\hat{g}(\mu_0)),\eta(\cdot|\mu',\hat{g}(\mu')))\leq  2K_f W_1(\mu_0,\mu').
\end{align*}
We can then write the following:
\begin{align*}
&|K^N_\beta(\mu_0,\hat{g})-\hat{K}_\beta^{N,*}(\mu_0)|\leq \int_{A_j}\left|k(\mu',\Theta_{\mu'})-k(\mu_0,\hat{g}(\mu_0))\right|\hat{\pi}_j(d\mu')\\
&+\left|\beta\int_{A_j}\int_{\mu_1}\hat{K}_\beta^{N,*}(\mu_1,g)\eta(d\mu_1|\mu',\Theta_{\mu'})\hat{\pi}_j(d\mu')-\beta\int_{A_j}\int_{\mu_1}K_\beta^{N,*}(\mu_1)\eta(d\mu_1|\mu',\Theta_{\mu'})\hat{\pi}_j(d\mu')\right|\\
&+\left|\beta\int_{A_j}\int_{\mu_1}K_\beta^{N,*}(\mu_1)\eta(d\mu_1|\mu',\Theta_{\mu'})\hat{\pi}_j(d\mu')-\beta\int K_\beta^{N,*}(\mu_1)\eta(d\mu_1|\mu_0,\hat{g}(\mu_0))\right|\\
&+\left|\beta\int K_\beta^{N,*}(\mu_1)\eta(d\mu_1|\mu_0,\hat{g}(\mu_0))-\beta\int K^N_\beta(\mu_1,\hat{g})\eta(d\mu_1|\mu_0,\hat{g}(\mu_0))\right|\\
&\leq 2K_cL_{\mathds{X}}+\beta\sup_{\mu}\left|\hat{K}_\beta^{N,*}(\mu)-K_\beta^{N,*}(\mu)\right|+\beta\|K_\beta^{N,*}\|_{Lip}2K_fL_{\mathds{X}}+\beta\sup_{\mu}\left|K^N_\beta(\mu,\hat{g})-K_\beta^{N,*}(\mu)\right|.
\end{align*}
Combining this bound and (\ref{first_bound}), and with an application of Proposition \ref{key_lem} and Lemma \ref{val_lip}, we get
\begin{align*}
(1-\beta)\sup_{\mu}\left|K^N_\beta(\mu,\hat{g})-K_\beta^{N,*}(\mu)\right|\leq 2K_cL_{\mathds{X}}+\frac{2K_c(1+\beta)}{(1-\beta)(1-2K_f\beta)}L_{\mathds{X}}+\frac{2K_c2K_f\beta}{1-2K_f\beta}L_{\mathds{X}}.
\end{align*}
Combining the terms, we get
\begin{align*}
\sup_{\mu}\left|K^N_\beta(\mu,\hat{g})-K_\beta^{N,*}(\mu)\right|\leq\frac{4K_c}{(1-\beta)^2(1-2K_f\beta)}L_{\mathds{X}}.
\end{align*}

\end{proof}

\section{Infinite Population Limit}\label{inf_pop}

In this section, we will look at the infnite population problem, i.e. the limit $N\to\infty$. We will first introduce the original problem for the infinite population, and then we will  analyze the infinite population limit for the model constructed in Section \ref{measure_valued_MDP}, i.e. the model defined on the discretized space. %The construction in this section will use a specially chosen weight measure (see (\ref{weight_meas})) for the ease of analysis. 

\subsection{Measure Valued MDP for the Infinite Population on $\mathds{X}$}

Similar to the construction in Section \ref{meas_valued_sec}, we will define a Markov decision process, where the control actions are the joint measures of the state and action pairs $(x,u)$

Different than Section \ref{meas_valued_sec}, we let the state space to be  $\mathcal{Z}=\P(\mathds{X})$ which is the set of all probability measures on $\mathds{X}$ (instead of the empirical measures).

The admissible set of actions for some state $\mu\in\mathcal{Z}$, is denoted by $U(\mu)$, where
\begin{align*}
U(\mu)=\{\Theta\in \P(\mathds{X}\times\mathds{U})|\Theta(\cdot,\mathds{{U}})=\mu(\cdot)\},
\end{align*}
that is, the set of actions for a state $\mu$, is the set of all joint probability measures on $\mathds{X}\times\mathds{U}$ whose marginal on $\mathds{X}$ coincides with $\mu$.

The stage-wise cost function is defined as before: for any $\mu\in\P(\mathds{X})$, and any admissible action $\Theta\in\P(\mathds{X}\times\mathds{{U}})$:
\begin{align*}
k(\mu,\Theta)=\int c(x,u,\mu)\Theta(dx,du)
\end{align*}
For the dynamics, we have that 
\begin{align*}
\mu_{t+1}=F(\mu_t,\Theta_t,w_t^0)
\end{align*}
where $w_t^0$ is the common noise, and
\begin{align*}
F(\mu_t,\Theta_t,w_t^0)=\mathcal{T}^{w^0_t}(\cdot|x,u,\mu_t)\Theta_t(dx,du)
\end{align*}
recall (\ref{kernel_common_noise}) for $\mathcal{T}^{w^0}$.

In particular, we define the infinite horizon discounted expected cost function under a policy $g$ by
\begin{align*}
K_\beta(\mu_0,g)=E_{\mu_0}^{g}\left[\sum_{t=0}^\infty \beta^t k(\mu_t,\Theta_t)\right].
\end{align*}

We also define the optimal cost by
\begin{align*}
K_\beta^{*}(\mu_0)=\inf_{g\in G}K_\beta(\mu_0,g).
\end{align*}

\subsection{Measure Valued MDP for the Infinite Population on $\hat{\mathds{X}}$}\label{inf_on_fin}
We now study the effect of the space discretization on the infinite population problem. We first introduce the limit problem on the finite space $\hat{\mathds{X}}$.

As before,  we start by choosing a collection of disjoint sets $\{B_i\}_{i=1}^M$ such that $\cup_i B_i=\mathds{X}$, and $B_i\cap B_j =\emptyset$ for any $i\neq j$. Furthermore, we choose a representative state, $\hat{x}_i\in B_i$, for each disjoint set. We denote the new finite state space by
$\hat{\mathds{X}}:=\{\hat{x}_1,\dots,\hat{x}_M\}$. 
 The mapping from the original state space to the finite set $\hat{\mathds{X}}$ is done via
\begin{align}\label{quant_map}
\phi(x)=\hat{x}_i \quad \text{ if } x\in B_i.
\end{align}

Unlike the finite population problem, for the dynamics and the cost function on the finite space, we will not perform an averaging, but instead we will use the transition function and the cost function of the representative state directly for the ease of analysis on the averaging over the measure space. Therefore,  we choose the weight measures on the discretization bins as the Dirac measure on the representative states, that is, for any $i$, $\pi_i(\cdot)=\delta_{\hat{x}_i}(\cdot)$.

With these weight measures, the discretized team control problem for the $N$ player on $\hat{\mathds{X}}$ becomes a special case of the model introduced in Section \ref{measure_valued_MDP}, such that the stage-wise cost function is given by
\begin{align*}
{\bf c}({\bf \hat{x}_t,u_t})=\frac{1}{N}\sum_{i=1}^N c(\hat{x}_t^i,u_i^t,\mu_{\bf{\hat{x}_t}})
\end{align*}
and the dynamics for some agent $i$ are given by
\begin{align*}
\hat{x}^i_{t+1}=\phi(f(\hat{x}_t^i,u_t^i,\mu_{\bf \hat{x}_t},w_t^i,w_t^0)).
\end{align*}
That is, with the chosen weight measures, instead of averaging over the quantization bins for the cost and the dynamics, we use the cost and the transition functions of the original model for the representative state of the corresponding bin.

We then have by Proposition \ref{key_lem}, under Assumption \ref{main_assmp}, for any $\mu\in\P_N({\mathds{X}})$
\begin{align*}
\left|\hat{K}_\beta^{N,*}(\mu)-K_\beta^{N,*}(\mu)\right|\leq \frac{2K_c}{(1-\beta)(1-2\beta K_f)}L_{\mathds{X}}.
\end{align*}
since the given bound does not depend on the choice of the weight measure $\pi$.

We now define the corresponding measure valued control problem on $\P(\hat{\mathds{X}})$. This construction will correspond to the infinite population limit of the team problem defined on $\hat{\mathds{X}}$.
We let the state space of the measure valued MDP to be $\P(\mathds{\hat{X}})$ which is the set of all probability measures on $\mathds{\hat{X}}$.

The admissible set of actions for some state $\hat{\mu}$, is denoted by $U(\hat{\mu})$, where
\begin{align*}
U(\hat{\mu})=\{\hat{\Theta}\in \P(\mathds{\hat{X}}\times\mathds{{U}})|\hat{\Theta}(\cdot,\mathds{{U}})=\hat{\mu}(\cdot)\}.
\end{align*}

For any $\hat{\mu}\in\P(\mathds{\hat{X}})$, and any admissible action $\hat{\Theta}\in\P(\mathds{\hat{X}}\times\mathds{{U}})$:
\begin{align*}
k(\hat{\mu},\hat{\Theta})=\int c(x,u,\hat{\mu})\hat{\Theta}(dx,du)
\end{align*}
Note that $\hat{\Theta}$ is a measure on a finite space, but we keep the integral sign instead of using the summation sign for notation consistency.

For the dynamics, we have that 
\begin{align*}
\hat{\mu}_{t+1}=\hat{F}(\hat{\mu}_t,\hat{\Theta},w_t^0)
\end{align*}
where $w_t^0$ is the common noise, and
\begin{align*}
\hat{F}(\hat{\mu}_t,\hat{\Theta}_t,w_t^0):=\hat{\mathcal{T}}^{w^0_t}(\cdot|x,u,\hat{\mu}_t)\hat{\Theta}_t(dx,du)
\end{align*}
such that for any $\hat{x}_i\in\hat{\mathds{X}}$
\begin{align*}
\hat{\mathcal{T}}^{w^0_t}(\hat{x}_i|x,u,\hat{\mu}_t)=\mathcal{T}^{w^0_t}(B_i|x,u,\hat{\mu}_t)
\end{align*}
where $B_i$ is the quantization bin of the representative state $\hat{x}_i$. This measure flow represents the flow of the marginal distribution of the state of any agent in the infinite population on the finite space $\hat{\mathds{X}}$.

 The infinite horizon discounted expected cost function under a policy $\hat{g}$  is defined by
\begin{align*}
\hat{K}_\beta(\hat{\mu}_0,\hat{g})=E_{\hat{\mu}_0}^{\eta,\hat{g}}\left[\sum_{t=0}^\infty \beta^t k(\hat{\mu}_t,\hat{\Theta}_t)\right].
\end{align*}

We also define the optimal cost by $\hat{K}_\beta^{*}(\hat{\mu}_0)=\inf_{\hat{g}\in \hat{G}}\hat{K}_\beta(\hat{\mu}_0,\hat{g})$.

\subsection{Effect of Discretization on the Value Function of the Infinite Population MDP}
As the final result of this section, we show that the discretization results in the same error bound for the infinite population problem, as in the finite population problem (Proposition \ref{key_lem}).
\begin{proposition}\label{inf_disc}
Under Assumption \ref{main_assmp}, for any $\hat{\mu}_0$ and $\mu_0$ such that $\hat{\mu}_0(\hat{x}_i)=\mu(B_i)$ for all $i$,
\begin{align*}
\left|\hat{K}_\beta^{*}(\hat{\mu}_0)-K_\beta^{*}(\mu_0)\right|\leq \frac{2K_c}{(1-\beta)(1-2\beta K_f)}L_{\mathds{X}}.
\end{align*}
\end{proposition}

\begin{proof}
We write
\begin{align*}
\left|\hat{K}_\beta^{*}(\hat{\mu}_0)-K_\beta^{*}(\mu_0)\right|\leq &\left|\hat{K}_\beta^{*}(\hat{\mu}_0)-\hat{K}_\beta^{N,*}(\hat{\mu}_0^N)\right|+\left|\hat{K}_\beta^{N,*}(\hat{\mu}_0^N)- K_\beta^{N,*}(\mu_0^N)\right|\\
&+\left| K_\beta^{N,*}(\mu_0^N) -K_\beta^{*}(\mu_0)\right|
\end{align*}
where $\mu_0^N,\hat{\mu}_0^N$ are empirical measures coming from, correspondingly, $\mu_0,\hat{\mu}_0$. It is proven in \cite{motte2022quantitative,motte2022mean, bauerle2021mean} that the first and the last terms converge to $0$. Furthermore, the second term is bounded by $\frac{2K_c}{(1-\beta)(1-2\beta K_f)}L_{\mathds{X}}$ by Proposition \ref{key_lem}, which does not depend on $N$, hence the proof is complete.
\end{proof}

\section{Further Approximations for Large N}\label{large_N}
So far, we have focused on approximations based on space discretization. In particular, we have presented approximate models by discretizing the state ($\mathds{X}$) and action ($\mathds{U}$) spaces of the agents. We have studied the regret bounds for  policies constructed using the finite space models, addressing both the $N$-player team problem and the infinite population limit ($N \to \infty$) problem.

The state space of the finite measure valued model, $\P_N(\hat{\mathds{X}})$, which is the set of all empirical measures on $\hat{\mathds{X}}$ with length $N$,  have  size $\frac{(M+N-1)!}{(N-1)!M!}$ where $M$ is the number of quantization bins. For the infinite population problem on the finite space, the state space of the measure valued problem becomes $\P_N(\hat{\mathds{X}})$ which is simply the simplex over the finite set $\mathds{\hat{X}}$. Therefore, although discretization simplifies the model, the problem remains complex for large populations ($N \to \infty$).

An immediate observation is that as $N$ increases, the state space of the measure-valued MDP grows significantly. Hence, going to the limit problem may not be appealing for the sole purpose of finding the optimal solution. However, another challenge for large, but finite $N$, lies in the coordination of agents when implementing the optimal policy. The centralized MDP construction for finite $N$, provides a recipe for the team action: given the distribution of the agents, say $\mu^N$, the goal is to decide on the state-action distribution of the agents, say $\Theta^N(du,dx)$. In order to realize $\Theta^N$, the agents need to be coordinated separately. In fact, the agent-level policies, in general turns out be non-symmetric. For instance, if agents follow a symmetric randomized policy, $\gamma^N(du|x,\mu^N)$, by disintegrating $\Theta^N(du,dx) = \gamma^N(du|x,\mu^N)\mu^N(dx)$, the resulting action distribution will differ from $\Theta^N(du,dx)$ with positive probability.

The limit, $N\to\infty$, problem helps address the coordination challenge among agents. In this case, the solution provides the agents a recipe for the optimal action, represented by $\Theta(du,dx) \in \mathcal{P}(\mathds{X} \times \mathds{U})$, when the agent distribution is given by some $\mu\in\P(\mathds{X})$. In the limit case, the agents can simply disintegrat $\Theta$, e.g. $\Theta(du,dx)=\gamma(du|x,\mu)\mu(dx)$, and use the randomized policy $\gamma$  symmetrically without coordination, which will achieve the optimal performance. Hence, if the optimal control problem is solved, the application of the optimal control is straightforward, since every agent can use the same policy,  provided they have access to the mean-field term $\mu$.

Therefore, although the $N\to\infty$ limit increases the complexity of the optimality problem,  the coordination burden between the agents decreases significantly as the limit solution enables a more straightforward application of the optimal control.

The final challenge, for the large $N$ (or $N\to \infty$) problem, is the observation of the state, i.e. the mean-field term, or the exact distribution of the agents. As the number of agents grows, it becomes more difficult to observe and use the exact distribution. For the particular case, when $N\to\infty$, the mean-field term is simply the marginal distribution of the state of the representative agent, which can be computed by the agent when there is no common noise affecting the dynamics, or when the common noise realizations are available. However, in general, one might need to control the system under partial or incorrect observations of the perfect mean-field term (perfect distribution of the agents).

In this section, we focus on the limit problem, and we aim  to simplify the optimality analysis.  We propose two approximation methods to address both space reduction and the incomplete feedback observation problem. Both methods will build on the finite space model introduced in Section \ref{finite_model_sec} for finite populations, and in Section \ref{inf_on_fin} for infinite populations.

For the first approximation method, we   construct a new finite problem, by directly aggregating the finite space infinite population problem based on possible distribution of $n$ agents, where $n$ is a manageable number. Note that, this is not the same as solving a team problem  with $n$ agents, but rather an averaging of the infinite population problem, over the possible distribution of $n$ agents. This approximate model provides symmetric policies, defined on  a finite space, for the  finitely many agent distributions. To implement this approximate policy, we  assume that we can observe the mean-field term exactly, say $\mu\in\P(\mathds{X})$, and map $\mu$ to the nearest $\mu^n$ (possible distribution of $n$ agents), and instructing the agents to use the symmetric policy $\gamma(du|\hat{x},\mu^n)$, where $\hat{x}$ is the discretized version of the state $x$.
%\begin{itemize}
%\item[i] Assume that we can observe the mean-field term exactly, and map $\mu$ to the nearest $\mu^n$ (possible distribution of $n$ agents), and tell the agents to use the symmetric policy $\gamma(du|\hat{x},\mu^n)$, where $\hat{x}$ is the discretized version of the state $x$.
%\item[ii] Instead of using the exact mean-field term, pick $n$ agents at random, and use their state distribution $\mu^n$ to construct the policy $\gamma(du|\hat{x},\mu^n)$.
%\end{itemize} 

For the second method, we focus on a finite model for the infinite population problem by sampling the population randomly and select $n$ random agents to observe and control the system. The finite model is constructed considering possible movements of $n$ agents in an environment with infinitely many agents. After solving for the approximate control policy, which will be symmetric for all agents,  instead of using the exact mean-field term, we  only observe the distribution of $n$ agents at random, and use their state distribution $\mu^n$ to construct the policy $\gamma(du|\hat{x},\mu^n)$. For the implementation of this policy, each agent can either use the same sub-sample as others or use their own independent sub-sample.

In both cases we will be able to provide regret bounds of the obtained approximate policies  when they are used in the original setting with continuous spaces by comparing them to the optimal performance  of the original problem.

%Our second approximate model construction will be based on the $n$ agent control problem directly, instead of averaging the infinite population problem as on the first method. This way, we will get a control law that maps, say $\mu^n$, to some $\Theta^n(du,dx)=\gamma^n(du|x,\mu^n)\mu^n(dx)$. The agent level policy $\gamma^n(du|x,\mu^n)$, as we mentioned before, is not optimal if it is used by the agents in the $n$ agents problem. We will study what happens, if the symmetric policy $\gamma^n(du|x,\mu^n)$ is used in the infinite population problem. For this method, we will be able to establish asymptotic consistency as $n\to \infty$. 

\subsection{Finite Model Construction via Direct Aggregation of Measures}\label{finite_inf_pop}
In this section, we present an approximation method for large populations, which is done by aggregating the measures, by considering the infinite population problem. The approximation will be done using the discretized state space $\hat{\mathds{X}}$, and the measures defined on it $\P(\hat{\mathds{X}})$. Since, $\P(\hat{\mathds{X})}$ is uncountable, being the simplex over $\mathds{\hat{X}}$, we will use all possible distributions of $n$ agents, which we denote by $\P_n(\hat{\mathds{X}})$. In the following construction, we will map the set $\P(\hat{\mathds{X}})$ to  $\P_n(\hat{\mathds{X}})$ with a nearest neighbor mapping, and treat the measures that are mapped to the same empirical measure as the same aggregate state, with a proper averaging. In other words, for any $\hat{\mu}^n_i\in  \P_n(\hat{\mathds{X}})$, the set
\begin{align*}
A_i=\{\hat{\mu}\in\P(\hat{\mathds{X}}):\rho(\hat{\mu})=\hat{\mu}^n_i\}
\end{align*}
will be treated as one state, where $\rho$ denotes the nearest neighbor map under the first order Wasserstein distance $W_1$, i.e.
\begin{align*}
W_1\left(\rho(\hat{\mu}),\hat{\mu}\right)\leq W_1\left(\hat{\mu}^n,\hat{\mu}\right)
\end{align*} 
for any $\hat{\mu}^n$.

We now construct a finite MDP for the infinite population, by defining the stage-cost and the transition functions. Different from the previous constructions, for the control action, we consider the randomized agent level policies $\gamma(du|x,\mu)$. Furthermore, we will not use an averaging over the aggregate sets.

For the cost function, we use
\begin{align*}
\hat{k}(\hat{\mu}^n,\gamma)&=\int_{x\in \hat{\mathds{X}}}\int_{u\in\mathds{U}} c(x,u,\hat{\mu}^n)\gamma(du|x,\hat{\mu}^n)\hat{\mu}^n(dx)\\
&=\sum_{\hat{x}_j,\hat{u}_i}c(\hat{x}_j,\hat{u}_i,\hat{\mu}^n)\gamma(\hat{u}_i|\hat{x}_j,\hat{\mu}^n)\hat{\mu}^n(\hat{x}_j)
\end{align*}
note that $\hat{\mathds{X}}$ and  $\mathds{U}$ are finite sets, however, we still sometimes use integration sign instead of summation for notation consistency.

For the transition function:
\begin{align}\label{trans_agg}
Pr(\hat{\mu}^n_i|\hat{\mu}^n_j,\gamma)=Pr\left(w_0: \rho\left(\int \mathcal{T}^{w_0}(\cdot|x,u,\hat{\mu}^n_j)\gamma(du|x,\hat{\mu}^n_j)\hat{\mu}^n_j(dx)\right)=\hat{\mu}^n_i\right)
\end{align}
in words, we consider the $\hat{\mu}_j^n$ as a usual probability measure on $\mathds{X}$, which is simply a combination of Dirac measures on the discrete points $\hat{x}_i$. We then look at the possible incoming measures under different common noise realizations, and we take the ones that will be mapped to $\hat{\mu}^n_i$ by the map $\rho$.

We denote the optimal cost for this model by $\hat{K}_\beta^*(\hat{\mu}^n)$ by overriding the notation.

\begin{remark}
We note that the aggregation of the measures can be done for the finite population problems as well without considering the infinite population limit, i.e. if the original number of agents is given by $N$, one can aggregate the measures $\hat{\mu}^N$ and map them to some $\hat{\mu}^n$.  This can also be done by rounding the number of agents on the quantization bins to get a cruder counting method, which also simplifies the distribution of the $N$ agents. This aggregation decreases the size of the state space from $\P_N(\hat{\mathds{X}})$ to $\P_n(\hat{\mathds{X}})$, however, the coordination of the agents is not simplified, since, without going to the infinite limit symmetric policies are not optimal. Furthermore, if we map the original distribution $\hat{\mu}^N$ to some $\hat{\mu}^n$, or count the agents by rounding, we should also keep track of how the agents are moved in order to map the original distribution, since we will need the inverse of this map to decide on the action distribution of the agents.

As an example, consider a team problem with 4 agents, where the state and action spaces are $\mathds{X,U}=\{0,1\}$. Say, the state distribution of the agents is given by $\mu(\cdot)=\frac{1}{4}\delta_0(\cdot)+\frac{3}{4}\delta_1(\cdot)$, that is 3 agent have state $1$ and one agent has state $0$. Further assume that we aggregate the distributions using possible distributions of 2 agents, and map the original state distribution to $\mu'(\cdot)=\delta_1(\cdot)$ (we might map it to $\frac{1}{2}\delta_0 + \frac{1}{2}\delta_1$ as well). If the action assigned to $\delta_1(\cdot)$ is $\frac{1}{2}\delta_{(1,0)}(\cdot) +\frac{1}{2}\delta_{(1,1)}(\cdot)$, then to apply this action in the original state, we need to record which agent is moved to state $1$, so that its action is also adjusted accordingly. 

Since the coordination complexity remains without going to the infinite population limit, we consider the infinite population to simplify the coordination challenge between the agents.

\end{remark}

\subsubsection{Approximate Control by Observing the Mean-field Term Perfectly}

We summarize the construction we will use in this section as follows:
\begin{itemize}
\item Construct and solve the finite model presented in Section \ref{finite_inf_pop}. This will provide an agent level policy $\gamma(du|\hat{x},\hat{\mu}^n)$, for every possible $\hat{\mu}^n$, and every $\hat{x}\in\mathds{X}$.
\item To apply this policy on the original model, observe the mean-field term, say $\mu\in\P(\mathds{X})$, find its counterpart say $\hat{\mu}$ defined on $\hat{\mathds{X}}$. This is done directly by defining $\hat{\mu}(\hat{x}_i)=\mu(B_i)$, for every  quantization bin $B_i\subset \mathds{X}$.   
\item Find the nearest possible empirical measure on $\mathds{\hat{X}}$, to $\hat{\mu}$, i.e. find $\rho(\hat{\mu})$.
\item Using the finite model solution, find the optimal policy for $\rho(\hat{\mu})$, say $\gamma(\cdot|\hat{x},\hat{\mu}^n)$.
\item The agents observe their state $x\in\mathds{X}$, find its discrete version $\hat{x}\in\hat{\mathds{X}}$, and find their control action via  $\gamma(\cdot|\hat{x},\hat{\mu}^n)$.
\end{itemize}

Before we present the result of this section, we introduce a new notation for the error term:
\begin{align*}
m_n=\sup_{\hat{\mu}\in\P(\hat{\mathds{X}})}W_1\left(\hat{\mu},\rho(\hat{\mu})\right)
\end{align*}
where $\rho$ is nearest neighbor map, that maps $\hat{\mu}$ to the closest empirical measure defined on $\hat{\mathds{X}}$. We note that on $\hat{\mathds{X}}$, we use the following metric: $\left|\hat{x}-\hat{y}\right|=0$ if $\hat{x}=\hat{y}$ and  $\left|\hat{x}-\hat{y}\right|=1$ if $\hat{x}\neq\hat{y}$. Hence, for some $\hat{\mu},\hat{\nu}\in \P(\hat{\mathds{X}})$:
\begin{align*}
W_1(\hat{\mu},\hat{\nu})=\sum_{\hat{x}\in\hat{\mathds{X}}} \left|\hat{\mu}(\hat{x})-\hat{\nu}(\hat{x})\right|
\end{align*}

We first present some useful lemmas:
\begin{lemma}\label{inf_agg}
Let Assumption \ref{main_assmp} hold with $2K_f\beta<1$. For any $\mu\in\P(\mathds{X})$, define $\hat{\mu}$ to be $\mu$'s projection on $\hat{\mathds{X}}$ and define $\hat{\mu}^n:=\rho(\hat{\mu})$. We then have
\begin{align*}
\left|\hat{K}_\beta^*(\hat{\mu}^n)-K_\beta^*(\mu)\right|\leq C\left(L_{\mathds{X}}+m_n\right)
\end{align*}
where $C$ is some constant that depends on the system parameters, and $L_{\mathds{X}}$ is the quantization error defined in (\ref{quant_error}).
\end{lemma}
\begin{proof}
The proof can be found in Appendix \ref{inf_agg_proof}.
\end{proof}

The following result establishes a Lipschitz bound for the value function of the infinite population problem defined on $\mathds{X}$.
\begin{lemma}\label{inf_lip}
Let Assumption \ref{main_assmp} hold with $2K_f\beta<1$. We then have that
\begin{align*}
\left|K_\beta^*(\mu)-K_\beta^*(\mu')\right|\leq \frac{2K_c}{1-2K_f\beta} W_1(\mu,\mu')
\end{align*}
for any $\mu,\mu'\in\P(\mathds{X})$.
\end{lemma}
\begin{proof}
Let $\mu^n$ and $\mu^{n'}$, empirical measures coming from the measures $\mu$ and $\mu'$. We write
\begin{align*}
&\left|K_\beta^*(\mu)-K_\beta^*(\mu')\right|\\
&\leq \left|K_\beta^*(\mu)- K_\beta^{n,*}(\mu^n)\right| + \left| K_\beta^{n,*}(\mu^n) -  K_\beta^{n,*}(\mu^{n'})\right| + \left| K_\beta^{n,*}(\mu^{n'}) - K_\beta^*(\mu')\right|\\
&\leq  \left|K_\beta^*(\mu)- K_\beta^{n,*}(\mu^n)\right| + \frac{2K_c}{1-2K_f\beta}W_1(\mu^n,\mu^{n'})+ \left| K_\beta^{n,*}(\mu^{n'}) - K_\beta^*(\mu')\right|
\end{align*}
where we have used Lemma \ref{val_lip} to bound the second term. Taking the limit $n\to\infty$ concludes the proof.
\end{proof}

We now present the main result of this section that gives an upper bound for the regret or the error due the application of the sub-optimal policy provided by the algorithm.
\begin{theorem}
Let Assumption \ref{main_assmp} hold, and let $\hat{\gamma}^n$ denote the agent level policy given by the algorithm presented in this section. Application of this algorithm in the original infinite population problem results is in the following regret:
\begin{align*}
K_\beta(\mu_0,\hat{\gamma}^n)-K_\beta^*(\mu_0)\leq C\left(L_{\mathds{X}}+m_n\right)
\end{align*}
where $C$ is some constant that depends on the system parameters, and $L_{\mathds{X}}$ is the quantization error defined in (\ref{quant_error}).
\end{theorem}
\begin{proof}
We start by writing
\begin{align}\label{init_term}
K_\beta(\mu_0,\hat{\gamma}^n)-K_\beta^*(\mu_0)&= K_\beta(\mu_0,\hat{\gamma}^n)-K_\beta^*(\mu_0)\pm \hat{K}_\beta^*(\hat{\mu}_0^n)\nonumber\\
&\leq \left| K_\beta(\mu_0,\hat{\gamma}^n) -  \hat{K}_\beta^*(\hat{\mu}_0^n)\right|+ \left| \hat{K}_\beta^*(\hat{\mu}_0^n) - K_\beta^*(\mu_0)\right|
\end{align}
the second term is bounded by Lemma \ref{inf_agg}. 

In what follows we will make use of the quantization map $\phi$ (see (\ref{quant_map})), which maps any $x\in\mathds{X}$ to the representative state of the quantization bin it belongs to, say some $\hat{x}_i\in\hat{\mathds{X}}$.
For the first term, we write
\begin{align}\label{mid_term}
&\left| K_\beta(\mu_0,\hat{\gamma}^n) -  \hat{K}_\beta^*(\hat{\mu}_0^n)\right|\nonumber\\
&\leq  \int c(x,u,\mu_0)\hat{\gamma}^n(du|\phi(x),\hat{\mu}_0^n)\mu_0(dx) - \int c(\hat{x},u,\hat{\mu}_0^n)\hat{\gamma}^n(du|\hat{x},\hat{\mu}_0^n)\hat{\mu}_0^n(d\hat{x})\nonumber\\
&\qquad\pm  \int c(\phi(x),u,\hat{\mu}^n_0)\hat{\gamma}^n(du|\phi(x),\hat{\mu}_0^n)\mu_0(dx)\nonumber\\
& + \beta \int K_\beta(\mu_1,\hat{\gamma}^n)Pr(d\mu_1|\mu_0,\hat{\gamma}^n) - \beta \int \hat{K}_\beta^*(\rho(\hat{\mu}_1))Pr(d\mu_1|\hat{\mu}_0^n,\hat{\gamma}^n)\nonumber\\
&\qquad\pm  \beta \int K_\beta^*(\mu_1)Pr(d\mu_1|\mu_0,\hat{\gamma}^n)\pm  \beta \int K_\beta^*(\mu_1)Pr(d\mu_1|\hat{\mu}^n_0,\hat{\gamma}^n)\nonumber\\
&\leq C\left(L_{\mathds{X}}+ m_n\right) + \beta \sup_\mu \left|\hat{K}_\beta^*(\hat{\mu}^n)-K_\beta^*(\mu)\right| + \beta \sup_\mu \left|K_\beta(\mu,\hat{\gamma}^n)-K_\beta^*(\mu)\right|\nonumber\\
&\qquad+  \left| \beta \int K_\beta^*(\mu_1)Pr(d\mu_1|\mu_0,\hat{\gamma}^n) -  \beta \int K_\beta^*(\mu_1)Pr(d\mu_1|\hat{\mu}^n_0,\hat{\gamma}^n)\right|
\end{align}
We now analyze the last term in detail. For a given realization of common noise $w^0$, and a function $f$ such that $\|f\|_{Lip},\|f\|_\infty\leq 1$:
\begin{align*}
&\bigg|\int f(x_1)\mathcal{T}^{w^0}(dx_1|x,u,\mu_0)\hat{\gamma}^n(du|\hat{x},\hat{\mu}_0^n)\mu_0(dx)-\int f(x_1)\mathcal{T}^{w^0}(dx_1|x,u,\hat{\mu}^n_0)\hat{\gamma}^n(du|\hat{x},\hat{\mu}_0^n)\hat{\mu}^n_0(dx)\bigg|\\
&\leq \bigg|\int f(x_1)\mathcal{T}^{w^0}(dx_1|x,u,\mu_0)\hat{\gamma}^n(du|\hat{x},\hat{\mu}_0^n)\mu_0(dx) - \int f(x_1)\mathcal{T}^{w^0}(dx_1|\phi(x),u,\hat{\mu}^n_0)\hat{\gamma}^n(du|\hat{x},\hat{\mu}_0^n)\mu_0(dx)\bigg|\\
&+ \bigg| \int f(x_1)\mathcal{T}^{w^0}(dx_1|\phi(x),u,\hat{\mu}^n_0)\hat{\gamma}^n(du|\hat{x},\hat{\mu}_0^n)\mu_0(dx) - \int f(x_1)\mathcal{T}^{w^0}(dx_1|x,u,\hat{\mu}^n_0)\hat{\gamma}^n(du|\hat{x},\hat{\mu}_0^n)\hat{\mu}^n_0(dx)\bigg|\\
& \leq K_f \left(L_{\mathds{X}} + m_n\right)+ W_1(\hat{\mu}_0,\hat{\mu}_0^n)\leq C \left(L_{\mathds{X}} +m_n\right)
\end{align*}
for some constant $C<\infty$. Hence, we have that 
\begin{align*}
 \left| \beta \int K_\beta^*(\mu_1)Pr(d\mu_1|\mu_0,\hat{\gamma}^n) -  \beta \int K_\beta^*(\mu_1)Pr(d\mu_1|\hat{\mu}^n_0,\hat{\gamma}^n)\right|\leq \|K_\beta^*\|_{Lip} C \left(L_{\mathds{X}} +m_n\right)
\end{align*}
Furthermore, $\|K_\beta^*\|_{Lip}$ is also bounded with Lemma \ref{inf_lip}.  Hence, combining (\ref{init_term}) and (\ref{mid_term}), we can conclude the proof with Lemma \ref{inf_agg}.
\end{proof}

\subsection{Finite Model Construction via Population Sampling}\label{finite_inf_pop2}
In this section, we present an approximation method for large populations, by considering the possible distribution of $n$ agents among the infinite population on the discrete space $\hat{\mathds{X}}$.

For the cost function, we use the same one as in the previous section:
\begin{align*}
\hat{k}(\hat{\mu}^n,\gamma)&=\int_{x\in \hat{\mathds{X}}}\int_{u\in\mathds{U}} c(x,u,\hat{\mu}^n)\gamma(du|x,\hat{\mu}^n)\hat{\mu}^n(dx)\\
&=\sum_{\hat{x}_j,\hat{u}_i}c(\hat{x}_j,\hat{u}_i,\hat{\mu}^n)\gamma(\hat{u}_i|\hat{x}_j,\hat{\mu}^n)\hat{\mu}^n(\hat{x}_j)
\end{align*}
note that $\hat{\mathds{X}}$ and  $\mathds{U}$ are finite sets, however, we still sometimes use integration sign instead of summation for notation consistency.

For the transition function:
\begin{align}\label{trans_samp}
Pr(\hat{\mu}^n_i|\hat{\mu}^n_j,\gamma)=\int Pr(\hat{\mu}^n_i|\hat{\mu}^n_j,\gamma,w^0)Pr(dw^0)
\end{align}
where
\begin{align*}
 Pr(\hat{\mu}^n_i|\hat{\mu}^n_j,\gamma,w^0)=Pr\left({\bf \hat{x}}^n: \mu_{\bf \hat{x}}^n=\hat{\mu}_i^n \text{ s.t. }, {\bf \hat{x}}^n\sim  \mathcal{T}^{w_0}(\cdot|x,u,\hat{\mu}^n_j)\gamma(du|x,\hat{\mu}^n_j))\hat{\mu}^n_j(dx)  \right)
\end{align*}
in words, we consider the $\hat{\mu}_j^n$ as a usual probability measure on $\mathds{X}$, which is simply a combination of Dirac measures on the discrete points $\hat{x}_i$. We then look at the possible incoming measures under different common noise realizations, and we look at the probability of an empirical distribution $\hat{\mu}_i^n$ to be produced by this measure.

We denote the optimal cost for this model by $\hat{K}_\beta^*(\hat{\mu}^n)$ by overriding the notation.

\subsubsection{Approximate Control by Observing the Mean-field Term via Population Sampling}

We summarize the construction we will use in this section as follows:
\begin{itemize}
\item Construct and solve the finite model presented in Section \ref{finite_inf_pop2}. This will provide an agent level policy $\gamma(du|\hat{x},\hat{\mu}^n)$, for every possible $\hat{\mu}^n$, and every $\hat{x}\in\mathds{X}$.
\item To apply this policy on the original model, pick $n$ agents, and look at their distributions on $\hat{\mathds{X}}$, say $\hat{\mu}^n$
\item The agents observe their state $x\in\mathds{X}$, find its discrete version $\hat{x}\in\hat{\mathds{X}}$, and find their control action via  $\gamma(\cdot|\hat{x},\hat{\mu}^n)$.
\end{itemize}

Before we present the result of this section, we introduce a new notation for the error term:
\begin{align*}
M_n=\sup_{\hat{\mu}\in\P(\hat{\mathds{X}})}E_{\hat{\mu}}\left[W_1\left(\hat{\mu},\hat{\mu}^n\right)\right]
\end{align*}
which is the expected distance of possible empirical measures produced by $\hat{\mu}$, to $\hat{\mu}$ itself.

We need the following lemma for the main result:
\begin{lemma}\label{pop_samp_lem}
Let Assumption \ref{main_assmp} hold with $2K_f\beta<1$. We then have that for any $\mu\in\P(\mathds{X})$,
\begin{align*}
E\left[\left|\hat{K}_\beta^*(\hat{\mu}^n)- K_\beta^*(\mu)\right|\right]\leq C \left(L_\mathds{X} + M_n\right)
\end{align*}
where the expectation is with respect to the possible realizations of the empirical distributions on $\hat{\mathds{X}}$, $\hat{\mu}^n$, coming from the original measure $\mu$.
\end{lemma}
\begin{proof}
The proof can be found in Appendix \ref{pop_samp_lem_proof}.
\end{proof}

The following result provides an upper bound for the regret or the error due the application of the sub-optimal policy provided by this algorithm.
\begin{theorem}
Let Assumption \ref{main_assmp} hold with $2K_f\beta<1$. Furthermore, let $\hat{\gamma}^n$ denote the agent level policy given by the algorithm presented in this section. Application of this algorithm in the original infinite population problem results is in the following regret:
\begin{align*}
K_\beta(\mu,\hat{\gamma}^n)-K_\beta^*(\mu)\leq C\left(L_{\mathds{X}}+M_n\right)
\end{align*}
where $C$ is some constant that depends on the system parameters, and $L_{\mathds{X}}$ is the quantization error defined in (\ref{quant_error}).
\end{theorem}

\begin{remark}
Note that the error term $M_n$ is closely related to the empirical consistency under Wasserstein distance, and one can find convergence rates in terms of $n$ for the measures defined on finite spaces in the literature.
\end{remark}

\begin{proof}
As before, we start by writing
\begin{align}\label{init_term2}
K_\beta(\mu,\hat{\gamma}^n)-K_\beta^*(\mu)&= K_\beta(\mu,\hat{\gamma}^n)-K_\beta^*(\mu)\pm E\left[\hat{K}_\beta^*(\hat{\mu}^n)\right]\nonumber\\
&\leq E\left[ K_\beta(\mu,\hat{\gamma}^n) -  \hat{K}_\beta^*(\hat{\mu}^n)\right]+ E\left[ \hat{K}_\beta^*(\hat{\mu}^n) - K_\beta^*(\mu)\right]
\end{align}
the second term is bounded by Lemma \ref{pop_samp_lem}.

Note that the application of of $\hat{\gamma}^n$ can be seen as a randomized policy for the original model, i.e., given the mean-field term $\mu$, an action, say $\hat{\gamma}^n(\cdot|\phi(x),\hat{\mu}^n)$ is chosen, with probability of observing $\hat{\mu}^n$ as the empirical measure under $\mu$. In what follows, we will use $Pr(\hat{\mu}^n|\mu)$ to denote this probability. 

For $K_\beta(\mu_0,\hat{\gamma}^n)$, we then have the following Bellman equation
\begin{align*}
K_\beta(\mu_0,\hat{\gamma}^n)=&\int c(x,u,\mu)\hat{\gamma}^n(du|\hat{x},\hat{\mu}^n)Pr(\hat{\mu}^n|\mu)\mu(dx)\\
&+\beta \int K_\beta(\mu_1,\hat{\gamma}^n)Pr(d\mu_1|\mu,\hat{\gamma}^n)
\end{align*}
where 
\begin{align*}
Pr(d\mu_1|\mu,\hat{\gamma}^n)&=\int \mathds{1}_{\{\mathcal{T}^{w^0}(\cdot|x,u,\mu)\hat{\gamma}^n(du|\hat{x},\hat{\mu}^n)\mu(dx)=\mu_1\}}Pr(\hat{\mu}^n|\mu)P(dw^0)\\
&=:Pr(d\mu_1|\mu,\hat{\gamma}(\hat{\mu}^n))P(\hat{\mu}^n|\mu)
\end{align*}

Thus, we can write the following:
\begin{align}\label{mid_term2}
&E\left[ K_\beta(\mu,\hat{\gamma}^n) -  \hat{K}_\beta^*(\hat{\mu}^n)\right]\nonumber\\
&\leq \int c(x,u,\mu)\hat{\gamma}^n(du|\hat{x},\hat{\mu}^n)\mu(dx)Pr(\hat{\mu}^n|\mu)- \int c(\hat{x},u,\hat{\mu}^n)\hat{\gamma}^n(du|\hat{x},\hat{\mu}^n)\hat{\mu}^n(d\hat{x})Pr(\hat{\mu}^n|\mu)\nonumber\\
&\qquad\pm  \int c(\hat{x},u,\hat{\mu}^n)\hat{\gamma}^n(du|\hat{x},\hat{\mu}^n)\mu(dx)Pr(\hat{\mu}^n|\mu)\nonumber\\
& + \beta \int K_\beta(\mu_1,\hat{\gamma}^n)Pr(d\mu_1|\mu,\hat{\gamma}^n(\hat{\mu}^n))Pr(\hat{\mu}^n|\mu) - \beta \int \hat{K}_\beta^*(\hat{\mu}^n_1)Pr(d\hat{\mu}^n_1|\hat{\mu}^n,\hat{\gamma}^n(\hat{\mu}^n))Pr(\hat{\mu}^n|\mu)\nonumber\\
&\qquad\pm  \beta \int K_\beta^*(\mu_1)Pr(d\mu_1|\mu,\hat{\gamma}^n(\hat{\mu}^n))Pr(\hat{\mu}^n|\mu)\pm  \beta \int K_\beta^*(\mu_1)Pr(d\mu_1|\hat{\mu}^n,\hat{\gamma}^n(\hat{\mu}^n))Pr(\hat{\mu}^n|\mu)\nonumber\\
&\leq C\left(L_{\mathds{X}}+ M_n\right) + \beta \sup_\mu E\left[\left|\hat{K}_\beta^*(\hat{\mu}^n)-K_\beta^*(\mu)\right|\right] + \beta \sup_\mu \left|K_\beta(\mu,\hat{\gamma}^n)-K_\beta^*(\mu)\right|\nonumber\\
&\qquad+  \left| \beta \int K_\beta^*(\mu_1)Pr(d\mu_1|\mu,\hat{\gamma}^n)Pr(\hat{\mu}^n|\mu) -  \beta \int K_\beta^*(\mu_1)Pr(d\mu_1|\hat{\mu}^n,\hat{\gamma}^n)Pr(\hat{\mu}^n|\mu)\right|
\end{align}
We now analyze the last term in detail. For a given realization of common noise $w^0$, and a function $f$ such that $\|f\|_{Lip},\|f\|_\infty\leq 1$:
\begin{align*}
&\bigg|\int f(x_1)\mathcal{T}^{w^0}(dx_1|x,u,\mu)\hat{\gamma}^n(du|\hat{x},\hat{\mu}^n)\mu(dx)Pr(\hat{\mu}^n|\mu)\\
&\qquad-\int f(x_1)\mathcal{T}^{w^0}(dx_1|x,u,\hat{\mu}^n)\hat{\gamma}^n(du|\hat{x},\hat{\mu}^n)\hat{\mu}^n(dx)Pr(\hat{\mu}^n|\mu)\bigg|\\
&\leq \bigg|\int f(x_1)\mathcal{T}^{w^0}(dx_1|x,u,\mu)\hat{\gamma}^n(du|\hat{x},\hat{\mu}^n)\mu(dx)Pr(\hat{\mu}^n|\mu) \\
&\qquad- \int f(x_1)\mathcal{T}^{w^0}(dx_1|\phi(x),u,\hat{\mu}^n)\hat{\gamma}^n(du|\hat{x},\hat{\mu}^n)\mu(dx)Pr(\hat{\mu}^n|\mu)\bigg|\\
&+ \bigg| \int f(x_1)\mathcal{T}^{w^0}(dx_1|\phi(x),u,\hat{\mu}^n)\hat{\gamma}^n(du|\hat{x},\hat{\mu}^n)\mu(dx)Pr(\hat{\mu}^n|\mu)\\
&\qquad - \int f(x_1)\mathcal{T}^{w^0}(dx_1|x,u,\hat{\mu}^n)\hat{\gamma}^n(du|\hat{x},\hat{\mu}^n)\hat{\mu}^n(dx)Pr(\hat{\mu}^n|\mu)\bigg|\\
& \leq C\left(L_{\mathds{X}} + M_n\right)
\end{align*}
for some constant $C<\infty$. Hence, we have that 
\begin{align*}
& \left| \beta \int K_\beta^*(\mu_1)Pr(d\mu_1|\mu,\hat{\gamma}^n)Pr(\hat{\mu}^n|\mu) -  \beta \int K_\beta^*(\mu_1)Pr(d\mu_1|\hat{\mu}^n,\hat{\gamma}^n)Pr(\hat{\mu}^n|\mu)\right|\\
&\leq \|K_\beta^*\|_{Lip} C \left(L_{\mathds{X}} +m_n\right)
\end{align*}
Furthermore, $\|K_\beta^*\|_{Lip}$ is also bounded with Lemma \ref{inf_lip}.  Hence, combining (\ref{init_term2}) and (\ref{mid_term2}), we can conclude the proof with Lemma \ref{inf_agg}.
\end{proof}

\subsection{Comparison of the Direct Aggregation and the Population Sampling Methods}
We have introduced two methods in this section; one based on direct aggregation of the probability measures (Section \ref{finite_inf_pop}) and another based on population sampling (Section \ref{finite_inf_pop2}).

\begin{itemize}
\item The regret bound presented for the aggregation method  (Section \ref{finite_inf_pop}) is smaller than the one on population sampling method (Section \ref{finite_inf_pop2}), since it is induced by the nearest neighbor map. 
\item Furthermore, for the aggregation method, to calculate the expectation in the Bellman equation via the transition kernel (equation (\ref{trans_agg})), for every given common noise realization, there is only one possible future state to consider. However, for the transition model of the sampling method (equation (\ref{trans_samp})), for every given common noise realization, the probability of each possible distribution of $n$ agents under the future measure needs to be considered to compute the expectation.  Hence the construction of the aggregation method in Section \ref{finite_inf_pop} is simpler.
\item However, for the application of the solution of the method in Section \ref{finite_inf_pop}, one needs to keep track of the perfect feedback, i.e. the state distribution of every agent, which might be hard to follow. For the application of the sampling method introduced in Section \ref{finite_inf_pop2}, instead of observing the perfect feedback, it is enough to observe the state information of a smaller sample of the original population.
\end{itemize}

\appendix

\section{Proofs of Technical Results}

\subsection{Proof of Lemma \ref{val_const}}\label{val_const_sec}

We will use value function approximation through value iterations. We define the sequence of functions $v_k:\mathds{X}^N\to\mathds{R}$ such that 
\begin{align*}
v_{k+1}({\bf x})=\inf_{{\bf u}\in\mathds{U}^N}\bigg({\bf c}({\bf x,u})+\beta E\big[ v_k({\bf X_1})|{\bf x,u}\big]\bigg)
\end{align*}
where $v_0({\bf x})=\inf_{{\bf u}\in\mathds{U}^N}{\bf c}({\bf x,u})$.

We first note that if ${\bf x}=(x^1\dots,x^N)$ , and ${\bf x'}=({x'}^1\dots,{x'}^N)$,  have the same distribution, $\mu_{\bf x}$, then they can be viewed as different orderings of the same state vector. Furthermore, for an action vector ${\bf u}=(u^1\dots,u^N)$, one can construct another action vector ${\bf u'}=({u'}^1\dots,{u'}^N)$, by reordering, such that the pairs $({\bf x,u})$ and $({\bf x',u'})$ have the same distribution. The immediate but key observation for the proof is that for the pairs  $({\bf x,u})$ and $({\bf x',u'})$, we have
\begin{align*}
{\bf c}({\bf x,u})=\frac{1}{N}\sum_{i=1}^Nc(x^i,u^i,\mu_{{\bf x}})=\frac{1}{N}\sum_{i=1}^Nc({x'}^i,{u'}^i,\mu_{{\bf x}})={\bf c}({\bf x',u'}).
\end{align*}
Also, for any noise vector ${\bf w}:=(w^0,w^1,\dots, w^N)$, the state vectors ${\bf x_1}=f({\bf x,u},\mu_{\bf x},{\bf w})$, and ${\bf x'_1}=f({\bf x',u'},\mu_{\bf x},{\bf w})$ (see (\ref{col_dyn})) will have the same distribution since the dynamics governing every agent are identical and hence the agents are exchangeable. In other words, 
\begin{align*}
\mu_{f({\bf x,u},\mu_{\bf x},{\bf w})}=\mu_{f({\bf x',u'},\mu_{\bf x},{\bf w})}
\end{align*}
if the pairs  $({\bf x,u})$ and $({\bf x',u'})$ have the same distribution.

Using these observations, we now prove the result with induction. For $v_0$, if ${\bf{x,x'}}\in \mathds{X}^N$ have the same distribution, $\mu_{\bf x}=\mu_{\bf x'}$, then we clearly have
\begin{align*}
\inf_{{\bf u}\in\mathds{U}^N}\frac{1}{N}\sum_{i=1}^Nc(x^i,u^i,\mu_{{\bf x}})=\inf_{{\bf u}\in\mathds{U}^N}\frac{1}{N}\sum_{i=1}^Nc({x'}^{i},u^i,\mu_{{\bf x'}}).
\end{align*} 
We now assume that the claim is true for $v_k$, i.e. if ${\bf{x,x'}}\in \mathds{X}^N$ have the same distribution, $\mu_{\bf x}=\mu_{\bf x'}$, then $v_k({\bf x})=v_k({\bf x'})$. Note that for ${\bf{x,x'}}\in \mathds{X}^N$ having the same distribution, $\mu_{\bf x}=\mu_{\bf x'}$
\begin{align*}
&v_{k+1}({\bf x})=\inf_{{\bf u}\in\mathds{U}^N}\bigg({\bf c}({\bf x,u})+\beta E\big[ v_k({\bf X_1})|{\bf x,u}\big]\bigg)\\
&v_{k+1}({\bf x'})=\inf_{{\bf u}\in\mathds{U}^N}\bigg({\bf c}({\bf x',u})+\beta E\big[ v_k({\bf X_1})|{\bf x',u}\big]\bigg).
\end{align*}
Suppose that $v_{k+1}({\bf x})<v_{k+1}({\bf x'})$, denoting the minimizer of the first equation by ${\bf u}$, (whose existence is guaranteed under Assumption \ref{main_assmp}), we construct ${\bf u'}$ such that the pairs  $({\bf x,u})$ and $({\bf x',u'})$ have the same distribution. We can the write
\begin{align*}
v_{k+1}({\bf x'})\leq{\bf c}({\bf x',u'})+\beta E\big[ v_k({\bf X_1})|{\bf x',u'}\big]&={\bf c}({\bf x',u'})+\beta \int v_k({\bf x_1})Pr(d{\bf x_1}|{\bf x',u'})\\
&={\bf c}({\bf x',u'})+\beta \int v_k(f({\bf x',u'},{\bf w}))P(d {\bf w})\\
&={\bf c}({\bf x,u})+\beta \int v_k(f({\bf x,u},{\bf w}))P(d {\bf w})\\
&=v_{k+1}({\bf x})
\end{align*}
where $P(dw)$ is the probability measure for the noise vector. We have used the symmetry of the cost function which implies that ${\bf c}({\bf x',u})={\bf c}({\bf x,u})$, and the induction step with the exchangeability of the agents to conclude that $v_k(f({\bf x',u'},{\bf w}))= v_k(f({\bf x,u},{\bf w}))$. Note that this follows from the fact that $f({\bf x',u'},{\bf w})$ and  $f({\bf x,u},{\bf w})$ have the same distribution, and the fact that $v_k$ is constant over the states with same distribution. Thus, we reach to a contradiction. Using a symmetrical argument, we can also conclude that $v_{k+1}({\bf x})>v_{k+1}({\bf x'})$ is not possible either, which implies that $v_{k+1}({\bf x})=v_{k+1}({\bf x'})$.

Finally, we write
\begin{align*}
|J_\beta^{N,*}({\bf x})-J_\beta^{N,*}({\bf x'})|&\leq |J_\beta^{N,*}({\bf x})-v_k({\bf x})|+|v_k({\bf x})-v_k({\bf x'})|+|v_k({\bf x'})-J_\beta^{N,*}({\bf x'})|\\
&\leq 2\|c\|_\infty\frac{\beta^k}{1-\beta}\to0,
\end{align*}
where we have used the fact that Bellman operator is a contraction under the uniform norm with modulus $\beta$ and its fixed point is the optimal value function.

\subsection{Proof of Theorem \ref{thm_meas}}\label{thm_meas_sec}
This result has been proved in \cite{bauerle2021mean}, however, we present a slightly different proof for completeness and because the proof method we use here will help us prove the other results in the paper. 

First, note that by Lemma \ref{val_const}, for any ${\bf x}_0$ that satisfies ${\bf x}_0=\mu_0$, $J_\beta^*({\bf x_0})$ has the same value. We now show that this value is also equal to the optimal value function of the measure valued MDP. 

We prove the result using value iterations. In particular, we will approximate, $J_\beta^{N,*}({x_0})$ and $K_\beta^{N,*}(\mu_0)$ using the iterations:
\begin{align*}
v_{k+1}({\bf x_0})&=\inf_{{\bf u}\in \mathds{U}^N}\bigg\{{\bf c(x_0,u)}+\beta E[v_k({\bf X_1})|{\bf x_0,u}]\bigg\},\\
w_{k+1}({ \mu_0})&=\inf_{\Theta\in U(\mu_0)}\bigg\{{ k(\mu_0,\Theta)}+\beta E[w_k({\bf \mu_1})|{\mu _0,\Theta}]\bigg\},
\end{align*}
where 
\begin{align*}
v_0({\bf x_0})&=\inf_{{\bf u}\in \mathds{U}^N} {\bf c(x_0,u)}\\
w_0({ \mu_0})&=\inf_{\Theta\in U(\mu_0)} { k(\mu_0,\Theta)}.
\end{align*}
We denote the minimizer of $v_0$ by ${\bf u^*}$, and we construct $\Theta$ such that $\Theta=\mu_{\bf (x_0,u^*)}$. We then have that
\begin{align*}
w_0({ \mu_0})\leq k(\mu_0,\Theta)=\frac{1}{N}\sum_{i=1}^Nc(x_0^i,u^{*,i},\mu_0)={\bf c(x_0,u^*)}=v_0({\bf x_0}).
\end{align*}
Conversely, let the minimizer for $w_0$ be $\Theta^*$. We pick some ${\bf u}$ such that $\mu_{\bf x_0,u}=\Theta^*$. We can then write
\begin{align*}
v_0({\bf x_0})\leq  {\bf c(x_0,u)}=\frac{1}{N}\sum_{i=1}^Nc(x_0^i,u^{i},\mu_0)=k(\mu_0,\Theta^*)=w_0({\mu_0}).
\end{align*}
Hence, we have that $v_0({\bf x_0})=w_0({\mu_0})$  for any ${\bf x}_0$ that satisfies ${\bf x}_0=\mu_0$. We now assume that  $v_k({\bf x_0})=w_k({\mu_0})$, and study  $v_{k+1}({\bf x_0})$ and $w_{k+1}({\mu_0})$.

We denote the minimizer of $v_k$ by ${\bf u^*}$, and we construct $\Theta$ such that $\Theta=\mu_{\bf (x_0,u^*)}$. We then have that
\begin{align*}
w_{k+1}({ \mu_0})&\leq k(\mu_0,\Theta)+\beta E[w_k({\bf \mu_1})|{\mu _0,\Theta}]\\
&= k(\mu_0,\Theta)+\beta\int w_k(\mu_1)\eta(d\mu_1|\mu_0,\Theta)\\
&= k(\mu_0,\Theta)+\beta\int \int_{{\bf x_1}:\mu_{\bf x_1}=\mu_1}w_k({ \mu_1})Pr(d{\bf x_1|x_0,u^*})\\
&={\bf c(x_0,u^*)}+\beta\int v_k({\bf x_1})Pr(d{\bf x_1}|{\bf x_0,u^*})\\
&=v_k({\bf x_0})
\end{align*}
where for the second equality, we used the definition of $\eta$, and for the third equality we used the induction argument.

Conversely, we let the minimizer for $w_k$ be $\Theta^*$. We pick some ${\bf u}$ such that $\mu_{\bf x_0,u}=\Theta^*$. We can then write
\begin{align*}
v_k({\bf x_0})&\leq {\bf c(x_0,u)}+\beta E[v_k({\bf X_1})|{\bf x_0,u}]\\
&=k(\mu_0,\Theta^*)+\beta E[w_k(\mu_1)|{\bf x_0,u}]\\
&=k(\mu_0,\Theta^*)+\beta E[w_k(\mu_1)|{\mu_0,\Theta^*}]\\
&=w_k(\mu_0)
\end{align*}
where for the first equality, we used the induction argument, and for the second equality, we used the fact that the probability distribution of $\mu_1$ is fully determined by the distribution of the pair $({\bf x_0,u})$. Hence, we have that $v_k({\bf x_0})=w_k({\mu_0})$  for any ${\bf x}_0$ that satisfies ${\bf x}_0=\mu_0$. Then, the first part of the result is completed by noting that 
\begin{align*}
|J_\beta^{N,*}({\bf x})-v_k({\bf x})|\leq \|c\|_\infty\frac{\beta^k}{1-\beta},\quad |K_\beta^{N,*}({\mu})-w_k({\mu})|\leq \|c\|_\infty\frac{\beta^k}{1-\beta}.
\end{align*}

For the second part of the result,  we first show the construction of the agent-level policy. For any given state vector ${\bf x}\in\mathds{X}^N$, $\Theta=g(\mu_{\bf x})$ assigns a distribution on $\mathds{X}\times\mathds{U}$ for the agents, hence the agents can choose their policies $\gamma^i:\mathds{X}\times\P_N(\mathds{X})\to \mathds{U}$ to realize this distribution. 

We now define the following Bellman operators, $T_\gamma$, and $\hat{T}_g$ for admissible policies $\gamma$ (possibly randomized) and $g$, such that for some measurable and bounded $v\in\M_b(\mathds{X}^N)$ and $w\in\M_b(\P_N(\mathds{X}))$
\begin{align*}
T_\gamma(v)({\bf x})&:={\bf c(x,\gamma)}+\beta E[v({\bf X_1})|{\bf x,\gamma}],\\
\hat{T}_g(w)(\mu)&:=k(\mu,g(\mu))+\beta E[w(\mu_1)|\mu,g(\mu)].
\end{align*}
where ${\bf c(x,\gamma)}:=\frac{1}{N}\sum_{i=1}^N\int_{\mathds{U}}c(x^i,u)\gamma(du|x^i)$ denotes the stage-wise cost function induced under the agent-level policy, which by construction is equal to $k(\mu_{\bf x},\Theta)$, and $E[v({\bf X_1})|{\bf x,\gamma}]$ denotes the conditional expectation induced under the agent-level policy $\gamma$.

Note that these operators are contraction under the uniform bound, thus $T_\gamma^k(v)({\bf x})\to J^N_\beta({\bf x},\gamma)$, and $\hat{T}_g^k(w)(\mu)\to K^N_\beta(\mu,g)$, where $T_\gamma^k$, $\hat{T}_g^k$ denote the resulting operators from $k$ consecutive application of the same operator.

Furthermore, if $v({\bf x})={\bf c(x,\gamma)}$ and $w(\mu_{\bf x})=k(\mu_{\bf x},g(\mu_{\bf x}))$, by construction of $\gamma$, using the identical arguments used for the proof of part (i), one can show that
\begin{align*}
T_\gamma^k(v)({\bf x})=\hat{T}_g^k(w)(\mu_{\bf x}), \text{ for all } k
\end{align*}
which proves that $J_\beta^N({\bf x},\gamma)=K^N_\beta(\mu_{\bf x},g)$ for all ${\bf x}\in\mathds{X}^N$ with same distribution.

Finally, the Bellman equation for $K_\beta^{N,*}(\mu_0)$:
\begin{align*}
K_\beta^{N,*}({ \mu_0})&=\inf_{\Theta\in U(\mu_0)}\bigg\{{ k(\mu_0,\Theta)}+\beta E[K_\beta^{N,*}({\bf \mu_1})|{\mu _0,\Theta}]\bigg\}.
\end{align*}
Under Assumption \ref{main_assmp}, the measurable selection conditions apply (\cite{brown1973measurable}), and there exists some $g:\P_N(\mathds{X})\to \P_N(\mathds{X}\times\mathds{U})$, which achieves the minimum cost. Hence, the result follows from part is complete.

\subsection{Proof of Theorem \ref{act_fin}}\label{act_fin_proof}

We start by writing the Bellman equations for the value functions:
\begin{align*}
&\tilde{J}_\beta^{N,*}({\bf x}_0)=\inf_{\hat{\bf u}\in\hat{\mathds{U}^N}}\left\{{\bf c(x_0,\hat{u})} + \beta E\left[\tilde{J}_\beta^{N,*}({\bf X}_1)| {\bf x}_0 , \hat{\bf u} \right] \right \}\\
&J_\beta^{N,*}({\bf x}_0)=\inf_{{\bf u}\in{\mathds{U}^N}}\left\{{\bf c(x_0,{u})} + \beta E\left[{J}_\beta^{N,*}({\bf X}_1)| {\bf x}_0 , {\bf u} \right] \right \}
\end{align*}

We denote the minimizes by $\hat{\bf u}$ and ${\bf u}$ for simplicity. Note that $\tilde{J}_\beta^{N,*}({\bf x}_0)\geq{J}_\beta^{N,*}({\bf x}_0)$ since the control space is smaller. If we denote by $\phi({\bf u})$ the coordinate-wise nearest neighbor map such that, each coordinate is mapped to the nearest element in $\hat{\mathds{U}}$, then application of $\phi({\bf u})$  results in a greater cost than $\hat{\bf u}$. We can then write
\begin{align*}
\tilde{J}_\beta^{N,*}({\bf x}_0)-{J}_\beta^{N,*}({\bf x}_0)&\leq  {\bf c(x_0,\phi({\bf u}))} +  \beta E\left[\tilde{J}_\beta^{N,*}({\bf X}_1)| {\bf x}_0 , \phi({\bf u}) \right]\\
& \quad-{\bf c(x_0,{u})} - \beta E\left[{J}_\beta^{N,*}({\bf X}_1)| {\bf x}_0 , {\bf u} \right]\\
&\quad \pm \beta E\left[{J}_\beta^{N,*}({\bf X}_1)| {\bf x}_0 , \phi({\bf u}) \right]
\end{align*}

For the cost function difference we have
\begin{align*}
 {\bf c(x_0,\phi({\bf u}))} -{\bf c(x_0,{u})} &= \frac{1}{N}\sum_{i=1}^N\left| c(x^i,\phi(u^i),\mu_{\bf x_0}) - c(x^i,u^i,\mu_{\bf x_0})\right|\\
&\leq  \frac{1}{N}\sum_{i=1}^N K_c \left|\phi(u^i)-u^i\right| \leq K_c L_{\mathds{U}}
\end{align*}
Furthermore, we have 
\begin{align*}
& \left| E\left[{J}_\beta^{N,*}({\bf X}_1)| {\bf x}_0 , {\bf u}\right]  -  E\left[{J}_\beta^{N,*}({\bf X}_1)| {\bf x}_0 , \phi({\bf u}) \right]\right| \\
&=  \left|\int K_\beta^{N,*}(\mu_{\bf f(x_0,u,w)})Pr(d{\bf w}) -  \int K_\beta^{N,*}(\mu_{\bf f(x_0,\phi(u),w)})Pr(d{\bf w})\right| \\
&\leq \| K_\beta^{N,*}\|_{Lip} W_1\left(\mu_{\bf f(x_0,u,w)}, \mu_{\bf f(x_0,\phi(u),w)} \right)
\end{align*}
$ \| K_\beta^{N,*}\|_{Lip}\leq \frac{2K_c \beta}{1-2K_f \beta}$ is bounded by Lemma \ref{val_lip}. For $W_1\left(\mu_{\bf f(x_0,u,w)}, \mu_{\bf f(x_0,\phi(u),w)} \right)$, since both are empirical measures, we know that the coupling achieves the Wasserstein distance is the one that pairs the state vectors ${\bf f(x_0,u,w)}$ and ${\bf f(x_0,\phi(u),w)}$ such that the average distance between them is minimized. In particular, we have that 
\begin{align*}
&W_1\left(\mu_{\bf f(x_0,u,w)}, \mu_{\bf f(x_0,\phi(u),w)} \right)\leq \frac{1}{N}\sum_{i=1}^N\left| f(x^i,\phi(u^i),\mu_{\bf x_0},w^i,w^0) - f(x^i,u^i,\mu_{\bf x_0},w^i,w^0)\right|\\
&\leq K_f L_{\mathds{U}}.
\end{align*}
Combining everything:
\begin{align*}
\tilde{J}_\beta^{N,*}({\bf x}_0)-{J}_\beta^{N,*}({\bf x}_0)&\leq  K_c L_{\mathds{U}} +\beta \sup_{\bf x} \left|\tilde{J}_\beta^{N,*}({\bf x})-{J}_\beta^{N,*}({\bf x})\right| + \frac{2K_cK_f\beta L_{\mathds{U}}}{1-2K_f\beta}
\end{align*}
which implies that 
\begin{align*}
\sup_{\bf x} \left|\tilde{J}_\beta^{N,*}({\bf x})-{J}_\beta^{N,*}({\bf x})\right|  \leq  \frac{K_c}{(1-2K_f\beta)(1-\beta)}L_{\mathds{U}}
\end{align*}

\subsection{Proof of Lemma \ref{cost_kernel_bound}}\label{cost_kernel_bound_proof}
Since $\mu,\mu'$ are empirical distributions, there always exist state vectors ${\bf x}=(x^1,\dots,x^N)$ and ${\bf \hat{x}}=(\hat{x}^1,\dots,\hat{x}^N)$ such that $\mu_{\bf x}=\mu,\mu_{\bf \hat{x}}=\mu'$. Furthermore, this holds for any permutation of the vectors ${\bf x,\hat{x}}$. If we denote by $\rho({\bf x})$ the possible permutations of ${\bf x}$, we then have $\mu_{\rho({\bf x}) }=\mu$ for every different permutation $\rho$.

The first order Wasserstein distance between $\mu$ and $\mu'$, can be written as 
\begin{align*}
W_1(\mu,\mu')=\inf_{P\in\Gamma(\mu,\mu')} E_P[|X-\hat{X}|]=\inf_{P\in\Gamma(\mu,\mu')}\sum_{j,k} |x^j-\hat{x}^k|P(X=x^j,\hat{X}=\hat{x}^k)
\end{align*}
where the infimum is over all couplings of $\mu$ and $\mu'$. The last equality follows as $\mu$ and $\mu'$ are the empirical measures induced by the vectors ${\bf x,x'}$ (or permutations of them). Following the last term, the coupling that achieves the minimum must concentrate on the closest pairings. Hence, we have that 
\begin{align}\label{min_was}
W_1(\mu,\mu')=\min_\rho\frac{1}{N}\sum_{i=1}^N|x^i-\rho(\hat{x}^i)|
\end{align}
where we fix  the vector ${\bf x}$ and the minimum goes over all possible  permutations, $\rho$, of the vector ${\bf \hat{x}}$. Since, there are finitely many different permutations, the minimum can be achieved. For the rest of the proof, we will use the state vectors ${\bf x,\hat{x}}$ that achieve the minimum.

For the cost function, using the fact that the empirical distribution of the control actions are identical and the assumption that $c(x,u,\mu)$ is Lipschitz in $x$ and $\mu$, we can write
\begin{align*}
&|k(\mu,\Theta)-k(\mu',\Theta')|=\bigg|\frac{1}{N}\sum_{i=1}^Nc(x^i,u^i,\mu)-c(\hat{x}^i,u^i,\mu')\bigg|\\
&\leq \frac{1}{N}\sum_{i=1}^N K_c\left(|x^i-\hat{x}^i|+W_1(\mu,\mu')\right)\\
&= 2K_c W_1(\mu,\mu')
\end{align*}
where $x^i$'s and $\hat{x}^i$'s are the elements of the state vectors that imply the empirical measures $\mu$ and $\mu'$ and that achieve the minimum in (\ref{min_was}). 

For the transition kernel, we will use a similar argument. We first note that by construction $\eta(\cdot|\mu,\Theta)=Pr(\mu_1\in\cdot|{\bf x,u})$ for any $\mu_{\bf x}=\mu$ and $\mu_{\bf x,u}=\Theta$. We can then write
\begin{align*}
&W_1(\eta(\cdot|\mu,\Theta),\eta(\cdot|\mu',\Theta'))=\sup_{Lip(h)\leq 1}\left|\int h(\mu_1)\eta(d\mu_1|\mu,\Theta)-\int h(\mu_1)\eta(d\mu_1|\mu',\Theta')\right|\\
&=\sup_{Lip(h)\leq 1}\left|\int h(\mu_1)\eta(d\mu_1|{\bf x,u})-\int h(\mu_1)\eta(d\mu_1|{\bf \hat{x},u})\right|\\
&=\sup_{Lip(h)\leq 1}\left|\int h(\mu_1({\bf x,u},\mu,{\bf w}))P(d{\bf w})-\int h(\mu_1({\bf \hat{x},u},\mu',{\bf w}))P(d{\bf w})\right|\\
&=\sup_{Lip(h)\leq 1}\left|\int h(\mu_1({\bf x,u},\mu,{\bf w}))- h(\mu_1({\bf \hat{x},u},\mu',{\bf w}))P(d{\bf w})\right|\\
&\leq \int W_1\left(\mu_1({\bf x,u},\mu,{\bf w}), \mu_1({\bf \hat{x},u},\mu',{\bf w})\right)P(d{\bf w})\\
%&=\int\sup_{Lip(h)\leq 1}\left|\int h({\bf x_1})\mu_1({\bf x,u},\mu,{\bf w})(d{\bf x_1})-\int h({\bf x_1})\mu_1({\bf \hat{x},u},\mu',{\bf w})(d{\bf x_1})\right|P(d{\bf w})\\
&=\int\sup_{Lip(h)\leq 1}\left|\frac{1}{N}\sum_{i=1}^N h(f(x^i,u^i,\mu,w^i))-h(f(\hat{x}^i,u^i,\mu',w^i))\right|P(d{\bf w})\\
&\leq \frac{1}{N}\sum_{i=1}^N K_f\left(|x^i-\hat{x}^i| + W_1(\mu,\mu')\right)\\
&=  2K_f W_1(\mu,\mu').
\end{align*}

\subsection{Proof of Lemma \ref{val_lip}}\label{val_lip_proof}
We start by writing the Bellman equation of the value functions
\begin{align*}
&K_\beta^{N,*}(\mu)=k(\mu,\Theta)+\beta\int_{\mu_1}K_\beta^{N,*}(\mu_1)\eta(d\mu_1|\mu,\Theta),\\
&K_\beta^{N,*}(\mu')=k(\mu',\Theta)+\beta\int_{\mu_1}K_\beta^{N,*}(\mu_1)\eta(d\mu_1|\mu',\Theta'),
\end{align*}
where we denote the optimal actions achieving the minimum on the right side of the Bellman equation by $\Theta$ and $\Theta'$ respectively for $\mu$ and $\mu'$, note that the existence of the minimizers is guaranteed under Assumption \ref{main_assmp}.

Note that $\mu,\mu',\Theta,\Theta'$ are empirical measures on $\mathds{X}$ and $\mathds{X}\times\mathds{U}$, hence we can find ${\bf x,x',u,u'}$, vectors with length N, such that $\mu_{\bf x}=\mu,\mu_{\bf x'}=\mu'$, $\mu_{\bf (x,u)}=\Theta,\mu_{\bf (x',u')}=\Theta'$ where ${\bf x,x'}$ are chosen in accordance with Lemma \ref{cost_kernel_bound}. %Furthermore, for the vectors ${\bf x,x'}$, we have that
%\begin{align*}
%\sup_{Lip(f)\leq 1}\left|\frac{1}{N}\sum_{i=1}^Nf(x^i)-f(x^{i'})\right|=W_1(\mu,\mu').
%\end{align*}

We will first assume that $K_\beta^{N,*}(\mu)>K_\beta^{N,*}(\mu')$. Since, $\Theta$ is the optimal action for the measure $\mu$ if we use $\theta=\mu_{\bf (x,u')}$, that is the distribution of the pair ${\bf x,u'}$ where ${\bf u'}$ is the action vector resulting in $\Theta'$, we then get a greater cost. Hence, we can write that 
\begin{align*}
K_\beta^{N,*}(\mu)-K_\beta^{N,*}(\mu')\leq &k(\mu,\theta)+\beta\int_{\mu_1}K_\beta^{N,*}(\mu_1)\eta(d\mu_1|\mu,\theta)\\
&- k(\mu',\Theta')-\beta\int_{\mu_1}K_\beta^{N,*}(\mu_1)\eta(d\mu_1|\mu',\Theta')\\
&\leq 2K_c W_1(\mu,\mu')+2K_f \beta\|K_\beta^{N,*}\|_{Lip}W_1(\mu,\mu'),
\end{align*}
where the last inequality follows from Lemma \ref{cost_kernel_bound}. Hence, by rearranging the terms, we get 
\begin{align*}
K_\beta^{N,*}(\mu)-K_\beta^{N,*}(\mu')\leq \frac{2K_c}{1-2K_f\beta}W_1(\mu,\mu').
\end{align*}
The case $K_\beta^{N,*}(\mu)<K_\beta^{N,*}(\mu')$ follows from identical steps.

\subsection{Proof of Proposition \ref{key_lem}}\label{key_lem_proof}
We first note that for some $\mu_j\in\P(\hat{\mathds{X}})$ the Bellman equation for the finite model value function can be written as
\begin{align*}
\hat{K}_\beta^{N,*}(\mu_j)=\inf_{\hat{\Theta}}\bigg\{\hat{k}(\mu_j,\hat{\Theta})+\beta\sum_{\mu_1}\hat{K}_\beta^{N,*}(\mu_1)\hat{\eta}(\mu_1|\mu_j,\hat{\Theta})\bigg\}.
\end{align*}
If we denote the minimizer for the right hand side by $\hat{\Theta}^*$, then when the value function is extended  over $\P_N(\mathds{X})$, by making it constant over the subsets $A_j$, we can write for any $\mu\in A_j$
\begin{align*}
\hat{K}_\beta^{N,*}(\mu)=\int_{A_j}k(\mu',\Theta_{\mu'})\hat{\pi}_j(d\mu')+\beta\int_{A_j}\int_{\mu_1}\hat{K}_\beta^{N,*}(\mu_1)\eta(d\mu_1|\mu',\Theta_{\mu'})\hat{\pi}_j(d\mu'),
\end{align*}
where for $\hat{\Theta}^*(dx,du)=\gamma(du|x)\mu_j(dx)$, $\Theta_{\mu'}(dx,du)=\gamma(du|\phi(x))\mu'(dx)$. 

Furthermore, the Bellman equation for the value function of the original model, $K_\beta^{N,*}(\mu)$, is
\begin{align*}
K_\beta^{N,*}(\mu)=k(\mu,\Theta^*)+\beta\int_{\mu_1}K_\beta^{N,*}(\mu_1)\eta(d\mu_1|\mu,\Theta^*),
\end{align*}
where we used $\Theta^*$ for the minimizer of the right side.

We first assume that $\hat{K}_\beta^{N,*}(\mu)>K_\beta^{N,*}(\mu)$. Note that the minimizer for the finite model is $\hat{\Theta}^*\in\P_N(\hat{\mathds{X}}\times\mathds{U})$, and we can find state and action vectors ${\bf \hat{x},\hat{u}}$ such that $\mu_{\bf (\hat{x},\hat{u})}=\hat{\Theta}^*$. The minimizer for the original model is $\Theta^*\in\P_N(\mathds{X}\times\mathds{U})$, and we can find state and action vectors ${\bf (x,u)}$  (in accordance with Lemma \ref{cost_kernel_bound})  such that $\mu_{\bf (x,u)}=\Theta^*$. Furthermore, if we use any other empirical measure on the space $\hat{\mathds{X}}\times\mathds{U}$ we will get a greater cost for the finite model. In particular, we define $\hat{\Theta}\in\P_N(\hat{\mathds{X}}\times\mathds{U})$, such that $\hat{\Theta}=\mu_{\bf (\hat{x},u)}$,
that is the marginal empirical distribution of the states on the finite set $\hat{\mathds{X}}$ stays the same, however, the  marginal empirical distribution of the actions are chosen to be the same as the ones coming from the minimizer action of the original model. By construction of $\Theta_{\mu'}$, we can also find ${\bf x',u}$ such that $\phi({\bf x'})={\bf \hat{x}}$, where $\phi$ is the discretization map, and $\mu_{\bf (x',u)}=\Theta_{\mu'}$.  We can then write:
\begin{align*}
&\hat{K}_\beta^{N,*}(\mu)-K_\beta^{N,*}(\mu)\leq \int_{A_j}k(\mu',\Theta_{\mu'})\hat{\pi}_j(d\mu')+\beta\int_{A_j}\int_{\mu_1}\hat{K}_\beta^{N,*}(\mu_1)\eta(d\mu_1|\mu',\Theta_{\mu'})\hat{\pi}_j(d\mu')\\
&\hspace{3cm}-k(\mu,\Theta^*)-\beta\int_{\mu_1}K_\beta^{N,*}(\mu_1)\eta(d\mu_1|\mu,\Theta^*).
\end{align*}
We now study the differences in the above term separately. 
\begin{align*}
&\int_{A_j}k(\mu',\Theta_{\mu'})\hat{\pi}_j(d\mu')-k(\mu,\Theta^*)\\
&\leq \int_{A_j}\left|k(\mu',\Theta_{\mu'})-k(\mu,\Theta^*)\right|\hat{\pi}_j(d\mu')\\
&\leq 2K_c \sup_{\mu'\in A_j}W_1(\mu,\mu')
\end{align*}
where the last inequality follows from Lemma \ref{cost_kernel_bound}, as  the empirical distributions of the control actions are identical.

For the second difference, we write
\begin{align*}
&\beta\int_{A_j}\int_{\mu_1}\hat{K}_\beta^{N,*}(\mu_1)\eta(d\mu_1|\mu',\Theta_{\mu'})\hat{\pi}_j(d\mu')-\beta\int_{\mu_1}K_\beta^{N,*}(\mu_1)\eta(d\mu_1|\mu,\Theta^*)\\
&=\beta\int_{A_j}\int_{\mu_1}\hat{K}_\beta^{N,*}(\mu_1)\eta(d\mu_1|\mu',\Theta_{\mu'})\hat{\pi}_j(d\mu')-\beta\int_{A_j}\int_{\mu_1}K_\beta^{N,*}(\mu_1)\eta(d\mu_1|\mu',\Theta_{\mu'})\hat{\pi}_j(d\mu')\\
&\qquad+\beta\int_{A_j}\int_{\mu_1}K_\beta^{N,*}(\mu_1)\eta(d\mu_1|\mu',\Theta_{\mu'})\hat{\pi}_j(d\mu')-\beta\int_{\mu_1}K_\beta^{N,*}(\mu_1)\eta(d\mu_1|\mu,\Theta^*)\\
&\leq \beta\sup_\mu \left|\hat{K}_\beta^{N,*}(\mu)-K_\beta^{N,*}(\mu)\right|+\|K_\beta^{N,*}\|_{Lip}\beta2K_f \sup_{\mu'\in A_j}W_1(\mu,\mu'),
\end{align*}
where the last line from Lemma \ref{cost_kernel_bound}. Thus, combining what we have so far:
\begin{align*}
\hat{K}_\beta^{N,*}(\mu)-K_\beta^{N,*}(\mu)&\leq \frac{2K_c+2K_f\beta\|K_\beta^{N,*}\|_{Lip}}{1-\beta}\sup_{\mu'\in A_j}W_1(\mu,\mu')\\
&\leq\frac{2K_c+2K_f\beta\|K_\beta^{N,*}\|_{Lip}}{1-\beta}L_{\mathds{X}}
\end{align*}
we used Lemma \ref{quant_lem} for the last step. The result then follows from Lemma \ref{val_lip}.

The case  $\hat{K}_\beta^{N,*}(\mu)<K_\beta^{N,*}(\mu)$ follows from almost identical steps. Hence, the proof is complete.

\subsection{Proof of Lemma \ref{inf_agg}}\label{inf_agg_proof}
We start by writing
\begin{align*}
\left|\hat{K}_\beta^*(\hat{\mu}^n)-K_\beta^*(\mu)\right|\leq &\left|\hat{K}_\beta^*(\hat{\mu}^n)-\hat{K}_\beta^*(\hat{\mu})\right|+\left|\hat{K}_\beta^*(\hat{\mu})-K_\beta^*(\mu)\right|
\end{align*}
where the second term is bounded by Proposition \ref{inf_disc}. For the first term, we write the following Bellman equations:
\begin{align*}
&\hat{K}_\beta^*(\hat{\mu}^n)=\int c(x,u,\hat{\mu}^n)\hat{\gamma}^n(du|x,\hat{\mu}^n)\hat{\mu}^n(dx) + \beta \int \hat{K}_\beta^*(\rho(\hat{\mu}_1))Pr(d\hat{\mu}_1|\hat{\mu}^n,\hat{\gamma}^n)\\
&\hat{K}_\beta^*(\hat{\mu})=\int c(x,u,\hat{\mu})\hat{\gamma}(du|x,\hat{\mu})\hat{\mu}(dx) + \beta \int \hat{K}_\beta^*(\hat{\mu}_1)Pr(d\hat{\mu}_1|\hat{\mu},\hat{\gamma})
\end{align*}
where $\hat{\gamma}^n(\cdot|x,\hat{\mu}^n)$ and $\hat{\gamma}(du|x,\hat{\mu})$ are the optimal agent level actions. To analyze the difference, we will use the same agent level policy for both models (say $\gamma(\cdot|x)$), as we did in previous proofs without loss of generality, and write
\begin{align*}
&\left|\hat{K}_\beta^*(\hat{\mu}^n)-\hat{K}_\beta^*(\hat{\mu})\right|\leq \left|\int c(x,u,\hat{\mu}^n)\gamma(du|x)\hat{\mu}^n(dx) -\int c(x,u,\hat{\mu})\gamma(du|x)\hat{\mu}(dx) \right|\\
&+ \bigg| \beta \int \hat{K}_\beta^*(\rho(\hat{\mu}_1))Pr(d\hat{\mu}_1|\hat{\mu}^n,\gamma)-\beta \int \hat{K}_\beta^*(\hat{\mu}_1)Pr(d\hat{\mu}_1|\hat{\mu},\gamma)\\
&\qquad\pm \beta \int \hat{K}_\beta^*(\hat{\mu}_1)Pr(d\hat{\mu}_1|\hat{\mu}^n,\gamma) \bigg|\\
&\leq \|c\|_\infty W_1\left(\hat{\mu}^n,\hat{\mu}\right) + \beta \|\hat{K}_\beta^*\|_{Lip} W_1\left(\hat{\mu}^n,\hat{\mu}\right) + \beta \sup_{\hat{\mu}}\left|\hat{K}_\beta^*(\rho(\hat{\mu})) - \hat{K}_\beta^*(\hat{\mu})\right|
\end{align*}
Hence, we can write that
\begin{align*}
\sup_{\hat{\mu}}\left|\hat{K}_\beta^*(\rho(\hat{\mu})) - \hat{K}_\beta^*(\hat{\mu})\right|& \leq \frac{\|c\|_\infty + \beta \|\hat{K}_\beta^*\|_{Lip}}{1-\beta}\sup_{\hat{\mu}}W_1\left(\hat{\mu}^n,\hat{\mu}\right)\\
&=\frac{\|c\|_\infty + \beta \|\hat{K}_\beta^*\|_{Lip}}{1-\beta}m_n
\end{align*}
hence the proof is complete as $\|\hat{K}_\beta^*\|_{Lip}$ can be shown to be bounded using similar arguments.

\subsection{Proof of Lemma \ref{pop_samp_lem}}\label{pop_samp_lem_proof}

The proof steps are similar to the proof of Lemma \ref{inf_agg}. We start by writing
\begin{align*}
\left|\hat{K}_\beta^*(\hat{\mu}^n)-K_\beta^*(\mu)\right|\leq &\left|\hat{K}_\beta^*(\hat{\mu}^n)-\hat{K}_\beta^*(\hat{\mu})\right|+\left|\hat{K}_\beta^*(\hat{\mu})-K_\beta^*(\mu)\right|
\end{align*}
where the second term is bounded by Proposition \ref{inf_disc}. For the first term, we write the following Bellman equations:
\begin{align*}
&\hat{K}_\beta^*(\hat{\mu}^n)=\int c(x,u,\hat{\mu}^n)\hat{\gamma}^n(du|x,\hat{\mu}^n)\hat{\mu}^n(dx) + \beta \int \hat{K}_\beta^*(\hat{\mu}^n_1))Pr(d\hat{\mu}_1^n|\hat{\mu}^n,\hat{\gamma}^n)\\
&\hat{K}_\beta^*(\hat{\mu})=\int c(x,u,\hat{\mu})\hat{\gamma}(du|x,\hat{\mu})\hat{\mu}(dx) + \beta \int \hat{K}_\beta^*(\hat{\mu}_1)Pr(d\hat{\mu}_1|\hat{\mu},\hat{\gamma})
\end{align*}
where $\hat{\gamma}^n(\cdot|x,\hat{\mu}^n)$ and $\hat{\gamma}(du|x,\hat{\mu})$ are the optimal agent level actions. To analyze the difference, we will use the same agent level policy for both models (say $\gamma(\cdot|x)$), and write
\begin{align*}
&\left|\hat{K}_\beta^*(\hat{\mu}^n)-\hat{K}_\beta^*(\hat{\mu})\right|\leq \left|\int c(x,u,\hat{\mu}^n)\gamma(du|x)\hat{\mu}^n(dx) -\int c(x,u,\hat{\mu})\gamma(du|x)\hat{\mu}(dx) \right|\\
&+ \bigg| \beta \int \hat{K}_\beta^*(\hat{\mu}_1^n)Pr(d\hat{\mu}^n_1|\hat{\mu}^n,\gamma)-\beta \int \hat{K}_\beta^*(\hat{\mu}_1)Pr(d\hat{\mu}_1|\hat{\mu},\gamma)\\
&\qquad\pm \beta \int \hat{K}_\beta^*(\hat{\mu}_1)Pr(d\hat{\mu}_1|\hat{\mu}^n,\gamma) \bigg|\\
&\leq \|c\|_\infty W_1\left(\hat{\mu}^n,\hat{\mu}\right) + \beta \|\hat{K}_\beta^*\|_{Lip} W_1\left(\hat{\mu}^n,\hat{\mu}\right) \\
&\qquad+ \beta\bigg|\int \hat{K}_\beta^*(\hat{\mu}_1^n)Pr(d\hat{\mu}^n_1|\hat{\mu}^n,\gamma)-\beta \int \hat{K}_\beta^*(\hat{\mu}_1)Pr(d\hat{\mu}_1|\hat{\mu}^n,\gamma) \bigg|
\end{align*}
for the last term, note that by construction $Pr(d\hat{\mu}^n_1|\hat{\mu}^n,\gamma)$ samples empirical measures from the probability measure $\mathcal{T}^{w^0}(\cdot|x,u,\hat{\mu}^n)\gamma(du|x)\hat{\mu}^n(dx)$ for every given common noise. Hence, for the last term we have that
\begin{align*}
\beta\bigg|\int \hat{K}_\beta^*(\hat{\mu}_1^n)Pr(d\hat{\mu}^n_1|\hat{\mu}^n,\gamma)-\beta \int \hat{K}_\beta^*(\hat{\mu}_1)Pr(d\hat{\mu}_1|\hat{\mu}^n,\gamma) \bigg|\leq \beta \sup_{\hat{\mu}} E\left[\left|\hat{K}_\beta^*(\hat{\mu}^n)-\hat{K}_\beta^*(\hat{\mu})\right|\right]
\end{align*}

Hence, we can write that
\begin{align*}
\sup_{\hat{\mu}}E\left[\left|\hat{K}_\beta^*(\hat{\mu}^n) - \hat{K}_\beta^*(\hat{\mu})\right|\right]& \leq \frac{\|c\|_\infty + \beta \|\hat{K}_\beta^*\|_{Lip}}{1-\beta}\sup_{\hat{\mu}}E\left[W_1\left(\hat{\mu}^n,\hat{\mu}\right)\right]\\
&=\frac{\|c\|_\infty + \beta \|\hat{K}_\beta^*\|_{Lip}}{1-\beta}M_n
\end{align*}
hence the proof is complete as $\|\hat{K}_\beta^*\|_{Lip}$ can be shown to be bounded using similar arguments.

%\section{Continuous Time mean-field Control}

%\bibliographystyle{plain}

\bibliographystyle{sn-basic}

%\bibliography{sn-bibliography}%

\bibliography{mfc_bibliography}

\section{Statements and Declarations}
\subsection{Funding}
E. Bayraktar is partially supported by the National Science Foundation under grant DMS-2106556 and by the Susan M. Smith chair.

\subsection{Competing Interests}
The authors have no relevant financial or non-financial interests to disclose.

\end{document}